\documentclass[12pt,reqno]{amsart}
\usepackage{amsfonts,amssymb,amsmath,amsopn,amsthm,graphicx}
\usepackage{amsxtra, mathrsfs}
\usepackage{colordvi}
\usepackage[usenames,dvipsnames]{color}
\usepackage{amsfonts,amssymb,amsbsy,amsmath,amsthm,dsfont}
\usepackage{xspace}

\usepackage{lipsum}
\usepackage[many]{tcolorbox}
\usetikzlibrary{decorations.pathreplacing}

 \numberwithin{equation}{section}

\newtcolorbox{leftbrace}{%
	enhanced jigsaw, 
	breakable, 
	frame hidden, 
	overlay={%
		\draw [
		decoration={brace,amplitude=0.5em},
		decorate,
		ultra thick,
		]
		(frame.south west)--(frame.north west);
	},
	parbox=false,
}

\usepackage{amssymb,amsmath,graphicx,amsthm,a4wide,wrapfig,caption,subcaption,epstopdf}
\usepackage[latin1]{inputenc}
\usepackage{enumitem}

\usepackage{verbatim}

\usepackage{lineno}
\usepackage{textcomp}
\usepackage{mathtools,hyperref}
\usepackage{cleveref}

\usepackage{setspace,esint}

\newcommand{\Hmm}[1]{\leavevmode{\marginpar{\tiny%
			$\hbox to 0mm{\hspace*{-0.5mm}$\leftarrow$\hss}%
			\vcenter{\vrule depth 0.1mm height 0.1mm width \the\marginparwidth}%
			\hbox to
			0mm{\hss$\rightarrow$\hspace*{-0.5mm}}$\\\relax\raggedright #1}}}

\newcommand{\bb}{{\mathbf{\bar {b}}}}
\newcommand{\bt}{{\mathbf{\tilde{b}}}}
\newcommand{\Gwb}{{\overline{\Omega}}}
\newcommand{\Hd}{\,H^1_{\Pw_{\mathrm{Dir}}}(\Gw)}

\newcommand{\Pw}{\partial \Omega}

\newcommand{\C}{\mathbb{C}}

\newcommand{\D}{{\rm D}}

\newcommand{\loc}{{\rm loc}}

\newcommand{\N}{\mathbb{N}}

\newcommand{\R}{\mathbb{R}}

\newcommand{\pwd}{\partial\Omega_{\mathrm{Dir}}}
\newcommand{\pwr}{\partial\Omega_{\mathrm{Rob}}}
\newcommand{\wb}{\overline{\Omega}\setminus \partial\Omega_{\mathrm{Dir}} }
\newcommand{\Dir}{\mathrm{Dir}}

\newcommand{\sign}{\mathrm{sign}\,}

\newtheorem{theorem}{Theorem}[section]

\newtheorem{cor}[theorem]{Corollary}
\newtheorem{thm}[theorem]{Theorem}
\newtheorem{lem}[theorem]{Lemma}
\newtheorem{lemma}[theorem]{Lemma}
\newtheorem{proposition}[theorem]{Proposition}

\newtheorem{defi}[theorem]{Definition}
\newtheorem{remark}[theorem]{Remark}

\newtheorem{rem}[theorem]{Remark}

\theoremstyle{definition}

\newtheorem{assumptions}[theorem]{Assumptions}

\numberwithin{equation}{section}
\newcommand{\diver}{\mathrm{div}\,}
\newcommand{\RN}[1]{%
	\textup{\uppercase\expandafter{\romannumeral#1}}%
}
\newcommand{\dx}{\,\mathrm{d}x}
\newcommand{\dy}{\,\mathrm{d}y}

\newcommand{\deta}{\,\mathrm{d}\eta}
\newcommand{\dsigma}{\,\mathrm{d}\sigma}

\newcommand{\be}{\begin{equation}}
\newcommand{\ee}{\end{equation}}
\newcommand{\bea}{\begin{eqnarray}}
\newcommand{\eea}{\end{eqnarray}}
\newcommand{\bean}{\begin{eqnarray*}}
	\newcommand{\eean}{\end{eqnarray*}}

\newcommand{\Rob}{\mathrm{Rob}}

\newcommand{\opname}[1]{\mbox{\rm #1}\,}
\newcommand{\supp}{\opname{supp}}

\newlength{\wex}  \newlength{\hex}
\newcommand{\ass}[1]{Let Assumptions~\ref{assump1} hold  in a bounded Lipschitz domain $\Gw$}

\def\ga{\alpha}     \def\gb{\beta}       \def\gg{\gamma}
       \def\gd{\delta}      
                         \def\vge{\varepsilon}
       \def\vgf{\varphi}    \def\gh{\eta}
            \def\gl{\lambda}
\def\gm{\mu}                 
    \def\gr{\rho}        
\def\gs{\sigma}       \def\gt{\tau}
      \def\gw{\omega}
                
\def\Gg{\Gamma}           

\def\Gl{\Lambda}          
\def\Gw{\Omega}              

\begin{document}

	\title[Criticality theory for mixed boundary value problems]{ On criticality theory for elliptic mixed boundary value problems in divergence form }
	

	\author{Yehuda Pinchover}
	
	\address{Yehuda Pinchover,
		Department of Mathematics, Technion - Israel Institute of
		Technology,   Haifa, Israel}
	
	\email{pincho@technion.ac.il}
	
	\author {Idan Versano}
	
	\address {Idan Versano, Department of Mathematics, Technion - Israel Institute of
		Technology,   Haifa, Israel}
	
	\email {idanv@campus.technion.ac.il}

	

	\begin{abstract}
		The paper is devoted to the study of positive solutions of a second-order linear elliptic equation in divergence form in a domain $\Gw\subseteq \R^n$ that satisfy an oblique boundary condition on a portion of $\partial \Gw$.  First, we study the degenerate mixed boundary value problem
			\begin{equation*}
		\begin{cases}
		Pu=f &  \text{in } \Omega, \\\tag{P,B}
		Bu = 0 &  \text{on } \Pw_{\mathrm{Rob}},\\
		u=0& \text{on } \pwd, 
		\end{cases}
		\end{equation*}
		where $\Gw$ is a bounded Lipschitz domain, $\pwr$ is a relatively open portion of $\partial \Gw$, $\pwd$ is a closed set of
		$\partial \Gw$, and $B$ is an oblique (Robin) boundary operator defined on $\pwr$.  In particular, we discuss the unique solvability of the above problem, the existence of a principal eigenvalue, and the existence of a positive minimal Green function. Then we establish a criticality theory for positive weak solutions of the operator $(P,B)$ in a general domain $\Gw$ with no boundary condition on  $\pwd$ and no growth condition at infinity. The paper generalizes and extends results obtained by Pinchover and Saadon (2002) for classical solutions of such a problem, where stronger regularity assumptions on the coefficients of $(P,B)$, and the boundary $ \Pw_{\mathrm{Rob}}$ are assumed.
		
		\medskip
		
		\noindent  2000  \! {\em Mathematics  Subject  Classification.}
			Primary  \! 35B09; Secondary  35J08, 35J70.\\[1mm]
	\end{abstract}
	
	\maketitle
	
	\section{Introduction}
	Let  $P$ be  a second-order, linear, elliptic operator, defined on a domain $\Gw\subseteq \R^n$, of the divergence form
		\begin{equation} \label{div_P}
		Pu:=-\diver\!\! \left[A(x)\nabla u +  u\bt(x) \right]  +
		\bb(x) \cdot \nabla u   +c(x)u \qquad x\in\Gw,
		\end{equation}
		with real measurable coefficients. Let $\pwr$ be a relatively open Lipschitz-portion of $\partial \Gw$, and consider the oblique boundary operator
		\begin{equation}\label{B_weak_for}
		Bu:=\beta(x) \big(A(x) \nabla u+u\bt(x)\big) \!\cdot\! \vec{n}(x) +\gamma(x) u \qquad x\in \pwr,
		\end{equation}
		where  $\xi \!\cdot \gh$ denotes the Euclidean inner product of the vectors  $\xi,\gh\in \R^n$, $\vec{n}(x)$ is the outward unit normal vector to $\partial \Gw$ at $x\in \Pw_{\mathrm{Rob}}$, and $\beta,\gamma$ are  real measurable functions defined on $\Pw_{\mathrm{Rob}}$.
The boundary  of $\Gw$ is then naturally decomposed to a disjoint union of its Robin part $\Pw_{\mathrm{Rob}}$, and its Dirichlet part $\Pw_{\mathrm{Dir}}\,$. That is,  $\partial\Omega=\Pw_{\mathrm{Rob}}\cup \Pw_{\mathrm{Dir}}$, where  $\Pw_{\mathrm{Rob}}\cap \Pw_{\mathrm{Dir}}=\emptyset$.
	
	The aim of the paper is to study properties of positive {\em weak} solutions of the equation $Pu=0$ in $\Gw$ satisfying weakly the oblique boundary conditions $Bu=0$ on $\pwr$, under minimal regularity assumptions on $(P,B)$ and $\partial\Gw$ (see Assumptions~\ref{assump2}).  In particular, we study the principal generalized eigenvalue, Green function, and the general properties of the cone of positive solutions. Such a study appears under the name {\em criticality theory}.
	
	The case $\pwr=\emptyset$ has been studied extensively in the past four decades, see for example \cite{Agmon,M1,P3}. Moreover, criticality theory for {\em classical solutions} of  the degenerate mixed boundary value problem 
		$$
		\begin{cases}
		-\sum\limits_{i,j=1}^{n}a^{ij}(x)D_{ij}u+
		\sum\limits_{i=1}^{n}b^i(x)D_i u+c(x)u
		=0  &\text{in } \Gw, \\
		\gamma(x)u+\beta(x)\frac{\partial u}{\partial \vec{n}}=0  &\text{on } \pwr,
		\end{cases}
		$$
		was established in \cite{PS} under the stronger regularity assumptions
		\begin{equation}\label{PPSS}
		a^{ij},b^i,c\in
		C^{\alpha}(\wb ),\; \pwr\in C^{2,\alpha},\; \gamma,\beta\in C^{1,\alpha}(\pwr), \beta>0, \gamma\geq 0.
		\end{equation}
		We note that the results in \cite{PS} relies heavily on the Hopf boundary point lemma which holds on $\pwr$ once the regularity assumptions \eqref{PPSS} are assumed. 
		 Furthermore, criticality theory for the adjoint operator $(P^*,B^*)$ was not discussed in \cite{PS}, and in particular, properties of the positive Green function $G^{\Gw}_{P,B}(x,y)$ as a function of $x$ and $y$ were not established.  
We mention also the related paper  \cite{Daners2}, where  D.~Daners studied  the case of a mixed {\em nondegenerate} boundary value problem in divergence form under the assumptions that $\Gw$ is a Lipschitz {\em bounded} domain, the coefficients of $P$ and $B$ are bounded, and both $\pwr$ and $\pwd$ are relatively open and closed subsets of  $\partial \Gw$. 
	
In the present paper we generalize the results obtained in \cite{PS} (and also in \cite{Daners2}) by passing from the realm of classical solutions to the realm of weak solutions  assuming  significantly weaker regularity assumptions (see Assumptions~\ref{assump1} for bounded Lipschitz domains, and Assumptions~\ref{assump2} for the general case). 
 We note that under these assumptions, the boundary point lemma does not necessarily hold. 
 Furthermore, our regularity assumptions on $\pwr$ force us to overcome a non-trivial geometric difficulty, namely,   the existence of  a bounded Lipschitz  exhaustion of $\wb$ (see Definition \ref{def:exhaustion}). Existence of such an exhaustion is known for the case where $\partial \Gw= \pwr\in C^3$ \cite{DR}.

	The paper is organized as follows.
		In Section~\ref{sec_prelim}, we introduce some necessary notation and assumptions, and define the notion of weak solutions to the problem 
		\begin{equation}\label{BPtag}
		\begin{cases}
		Pu=0 &  \text{in } \Omega, \\
		Bu = 0 &  \text{on } \Pw_{\mathrm{Rob}}.
		\end{cases}
		\end{equation}
		Section~\ref{sec_local} is devoted to the local theory. In particular, we study the coercivity  of the bilinear form associated to the mixed boundary value problem \eqref{BPtag} in the appropriate functional space $H^{1}_{\Pw_{\mathrm{Dir}}}(\Omega)$,  where $\Gw$ is a bounded  domain satisfying Assumptions~\ref{assump1}. In addition, we recall some known local regularity results needed for the rest of the paper, and discuss the compactness of the resolvent operator, the generalized maximum principle, the existence of a principal eigenvalue, and the Harnack convergence principle. In Section~\ref{sec_crit},  we develop a criticality theory when $(P,B)$ and $\Gw$ satisfy Assumptions~\ref{assump2}. More precisely, we define criticality/subcriticality of the operator $(P,B)$ in $\Gw$, obtain characterizations of subcritical and critical operators, and prove the existence of a ground state for critical operators. Section~\ref{sec_green} is devoted to the construction of the positive (minimal) Green function for a subcritical operator $(P,B)$ in $\Gw$. Finally, in Section~\ref{sec_Symm}, we discuss the symmetric case, where $\bt=\bb$, and in particular, prove the appropriate Allegretto-Piepenbrink-type theorem (cf. \cite{Agmon,P3,S}).
\section{Preliminaries and  notations}\label{sec_prelim}
Let $\Omega$ be a domain in $\R^n$, $n\geq 2$, and let $\partial\Omega=\Pw_{\mathrm{Rob}}\cup \Pw_{\mathrm{Dir}}$, where  $\Pw_{\mathrm{Rob}}\cap \Pw_{\mathrm{Dir}}=\emptyset$. We assume that $\Pw_{\mathrm{Rob}}$, the Robin-portion of $\partial\Omega$, is a relatively open subset of $\partial \Gw$, and $\Pw_{\mathrm{Dir}}$, the Dirichlet part of $\Pw$, is a closed set of $\partial \Gw$. Moreover, if $\Gw$ is a bounded domain, we further assume that in the relative topology of $\partial \Gw$, we have $\mathrm{int}( \Pw_{\mathrm{Dir}})\neq\emptyset$.
Throughout the paper we use the following notation and conventions:
\begin{itemize}

	\item For any $\xi\in \R^n$ and a positive definite symmetric matrix $A\in \R^{n \times n}$, let  
	$| \xi|_A:=\sqrt{A\xi\cdot \xi  }$, where $\xi\cdot\gh$ denotes the Euclidean inner product of $\xi,\gh\in \R^n$.
	\item From time to time we use the Einstein summation convention.
	
	\item The gradient of a function $f$ will be denoted either by $\nabla f$ or $Df$.
\item For  $x\in \R^n$, we denote $x=(x',x_n)$, where $x_n\in \R$.
\item For $R\!>\!0$ and $x\! \in\! \R^n$, we denote by $B_R(x)$ the open ball of radius $R$ centered at $x$.
\item $\chi_B$ denotes the characteristic function of a set $B\subset \R^n$.	
	\item We write $A_1 \Subset A_2$ if  $\overline{A_1}$ is a compact set, and $\overline{A_1}\subset A_2$.
		\item Let $\Gw$ be a domain and let $\Gw'$ be a subdomain. We write $\Gw'\Subset_R \Gw$ if \\ $\Gw'\Subset \overline{\Gw}$,
		$\partial \Gw' \cap \pwd =\emptyset$, and $\partial \Gw' \cap \pwr \Subset \pwr$ with respect to the relative topology on $\pwr$.  
\item For a subdomain  $\omega \! \Subset_R \!\Gw$, we define $\partial\omega_{\mathrm{Rob}}\! :=\!\mathrm{int}(\partial \omega \cap \pwr)$, and 
$\partial \omega_{\mathrm{Dir}} \!:=\! \partial \omega \!\setminus\!  \partial\omega_{\mathrm{Rob}}$. 	
\item For any $1\leq p\leq \infty$, $p'$ is the H\"older conjugate exponent of $p$ satisfying $p'=p/(p-1)$.
\item For $1\leq p<n$,  $p^*:=np/(n-p)$ is the corresponding Sobolev critical exponent.
\item For a Banach space $V$ over $\R$, we denote by $V^*$ the space of continuous linear maps from  $V$ into $\R$.
	\item $C$ refers to  a positive constant which may vary from  line to line.
	\item Let $g_1,g_2$ be two positive functions defined in $\Gw$. We use the notation $g_1\asymp g_2$ in
	$\Gw$ if there exists a positive constant $C$ such
	that
	$$C^{-1}g_{2}(x)\leq g_{1}(x) \leq Cg_{2}(x) \qquad \mbox{ for all } x\in \Gw.$$	
	\item Let $g_1,g_2$ be two positive functions defined in $\Gw$, and let $x_0\in \Gw$. We use the notation $g_1\sim g_2$ near $x_0$ if there exists a positive constant $C$ such that
	$$
	\lim_{x\to x_0}\frac{g_1(x)}{g_2(x)}=C.
	$$
	
	\item For any real measurable function $u$ and $\omega\subset \R^n$, we denote
	$$\inf_{\omega}u:=\mathrm{ess}\inf_{\omega}u, \quad \sup_{\omega}u:=\mathrm{ess}\sup_{\omega}u, \quad  u^+:=\max(0,u), \quad u^-:=\max(0,-u).$$
	\item $\sign u(x)= u(x)/|u(x)|$ if $u(x) \neq 0$, and  $\sign u(x)= 0$ if $u(x)=0$.      
\end{itemize}
Let $R,K>0$, and let $w$ be a real-valued Lipschitz continuous function defined on $B'_R:=\{x'\in \R^{n-1}:|x'|<R\}$ with 
$$
|w(x')-w(y')|\leq K|x'-y'| \qquad    \forall x',y'\in B'_R(0),
$$
and  $w(0)\in (0,R)$. We denote
\begin{align*}				
\Gw[R]=\{x\in \R^n: x_n>w(x'), \; |x|<R \}, \\ 
\sigma[R]=\{x\in \R^n: x_n \geq w(x'), \;|x|=R \}, \\ 
\Sigma[R]=\{x\in \R^n: |x|<R, \; x_n=w(x') \}.
\end{align*}
\begin{defi}[Lipschitz and $C^1$-portions]
	{\em 
Let $x_0\in \Pw$ and $R>0$ such that $\Gw[x_0,R]:=\Gw\cap B_R(x_0)$ is a Lipschitz (resp.,  $C^1$) domain. The set 
$\Sigma[x_0,R]=\Pw \cap  B_R(x_0)$ is called a {\em Lipschitz} (resp.,  $C^1$){\em-portion} of $\Pw$.
} 
\end{defi}
Further, we introduce some functional spaces. Denote 
$\mathcal{D}(\Omega,\Pw_{\mathrm{Dir}}):=  C_0^{\infty}(\overline{\Omega} \setminus \pwd)$. So,  $u\in\mathcal{D}(\Omega,\Pw_{\mathrm{Dir}})$ if $u$ has compact support and 
$$\supp u:=\overline{\{x\in \Gw\mid u(x)\neq 0\}} \subset \overline{\Omega} \setminus \pwd.$$  
For $q\geq 1$, we define $ W^{1,q}_{\Pw_{\mathrm{Dir}}}(\Omega)$ to be the closure of 
$\mathcal{D}(\Omega,\Pw_{\mathrm{Dir}})$  with respect to the Sobolev norm of $W^{1,q}(\Omega)$.
We also consider the following spaces: 
$$
L^{q}_{\loc}(\wb):=\{ u \mid  \forall x\in \Gw\cup \Pw_{\mathrm{Rob}},  \;   \exists r_x>0 \mbox{ s.t. } u\in L^q(\Omega\cap B_{r_x}(x))   \},$$	
	$$
 W^{1,q}_{\loc}(\overline{\Omega}\setminus \pwd):=\{ u\mid \forall x\in \Gw\cup \Pw_{\mathrm{Rob}},  \;   \exists r_x>0 \mbox{ s.t. } u\in W^{1,q}(\Omega\cap B_{r_x}(x))   \}.
 $$
In the case $q=2$ we omit the index $q$ and write
$H^{1}_{\loc}(\overline{\Omega}\setminus \pwd):=W^{1,2}_{\loc}(\overline{\Omega}\setminus \pwd)$.
\begin{rem}
	\em{
	For every Lipschitz subdomain $\Gw'\Subset_R \Gw$ and $1< q<\infty$
	the space $W^{1,q}(\Gw')$ is a reflexive Banach space and therefore,  $W^{1,q}_{\Pw_{\mathrm{Dir}}}(\Omega')$
	is reflexive as well. 
	}
\end{rem}
\medskip

Consider an elliptic operator  $P$ of the form \eqref{div_P} and a Robin boundary operator $B$ of the form \eqref{B_weak_for}. 
Throughout the paper we assume  the following regularity assumptions on $P$, $B$ and $\partial \Gw$:
\begin{leftbrace}\begin{assumptions}
		\label{assump2} 
		\begin{itemize}
			\item[{\ }]		
			\item $\partial\Omega=\Pw_{\mathrm{Rob}}\cup \Pw_{\mathrm{Dir}}$, where  $\Pw_{\mathrm{Rob}}\cap \Pw_{\mathrm{Dir}}=\emptyset$, and $\Pw_{\mathrm{Rob}}$ is a relatively open subset of $\Pw$.
			\item For each $x_0\in \Pw_{\mathrm{Rob}}$ there exists $R>0$ such that $\Sigma[x_0,R]$ is a $C^1$-portion.
			\item $A\!=\!(a^{ij})_{i,j=1}^{n}\in L_\loc^\infty(\Gwb \setminus \Pw_{\mathrm{Dir}}; \R^{n\times n})$ is a symmetric positive definite  matrix valued function which is 
			locally uniformly elliptic  in $\Gwb \setminus \Pw_{\mathrm{Dir}}$, that is, for any compact $K\subset \Gwb \setminus \Pw_{\mathrm{Dir}}$ there exists  $\Theta_K>0$ such that 
			\begin{eqnarray*} 
					\Theta_K^{-1}\sum_{i=1}^n\xi_i^2\leq\sum_{i,j=1}^n
				a^{ij}(x)\xi_i\xi_j\leq \Theta_K\sum_{i=1}^n\xi_i^2 \quad \forall \xi\in \mathbb{R}^n \mbox{ and } \forall x\in K.
			\end{eqnarray*}	
			\item $\bt, \bb\in  L^{p}_{\loc}(\wb;\R^n)$, and $c\in L^{p/2}_{\loc}(\wb)$ for some $p>n$.
		
			\item $\beta>0$, and $\gamma/\beta\in L^{\infty}_{\mathrm{loc}}(\Pw_{\mathrm{Rob}})$.	
		\end{itemize}
	
	\end{assumptions}
\end{leftbrace}
Moreover, sometimes we need to assume for {\em bounded domains} the following regularity requirements:
\begin{leftbrace}\begin{assumptions}
		\label{assump1}
	\begin{itemize}
		\item[{\ }]
	\item $\Gw$ is a bounded Lipschitz domain, and $\mathrm{int}( \Pw_{\mathrm{Dir}})\neq\emptyset$ in the relative topology of $\partial \Gw$.
	\item $\Pw_{\mathrm{Rob}}\subset \partial \Gw$ is a relatively open and locally Lipschitz subset of $\partial \Gw$.
	\item $\Pw_{\mathrm{Dir}}\subset \partial \Gw$ is a finite disjoint union of closures of Lipschitz-portions. 
	\item $A\!=\!(a^{ij})_{i,j=1}^{n}\in L^\infty(\Gwb; \R^{n\times n})$ is a symmetric positive definite  matrix valued function which is 
	uniformly elliptic in $\bar\Gw$, that is, there exists  $\Theta>0$ such that 
	\begin{eqnarray*} 
		\Theta^{-1}\sum_{i=1}^n\xi_i^2\leq\sum_{i,j=1}^n
		a^{ij}(x)\xi_i\xi_j\leq \Theta\sum_{i=1}^n\xi_i^2 \quad \forall \xi\in \mathbb{R}^n \mbox{ and } \forall x\in \bar\Gw.
	\end{eqnarray*}	
	\item $\bt, \bb\in L^p(\Gw;\R^n)$, and $c\in L^{p/2}(\Gw)$ for some $p>n$.
	\item $\beta>0$, and $\gamma/\beta\in L^{\infty}(\Pw_{\mathrm{Rob}})$.			
\end{itemize}

\end{assumptions}
\end{leftbrace}
\begin{rem}
	\em{
		The $C^1$-smoothness of $\pwr$ in Assumptions~\ref{assump2} is needed only in two parts of the paper: in  the construction of the Green function (see Section~\ref{sec_green}), and in the construction of  a Lipschitz-exhaustion of $\wb$ (See Appendix~\ref{appendix1}). In the rest of the paper it is enough to assume that $\pwr$ is locally Lipschitz. 
	} 
\end{rem}
\begin{rem}
	\em{
		If Assumptions~\ref{assump2} hold in $\Gw$, then Assumptions~\ref{assump1} hold in any Lipschitz  subdomain $\gw\Subset_R \Gw$  with $\partial \omega_{\mathrm{Dir}}= \partial \omega\cap \Gw$.
	} 
\end{rem}

Next, we define (weak) solutions and supersolutions of the boundary value problem
\begin{equation}\label{P,B}
\begin{cases}
Pu=0 &  \text{in } \Omega, \\\tag{{\bf P,B}}
Bu = 0 &  \text{on } \Pw_{\mathrm{Rob}}.
\end{cases}
\end{equation}
\begin{defi}{\em
		We say that $u\in H^{1}_{\loc}(\wb)$ is a  {\em weak solution (resp., supersolution)} of the  problem  \eqref{P,B} in $\Gw$, if for any (resp., nonnegative) $\phi\in \mathcal{D}(\Omega,\Pw_{\mathrm{Dir}})$   we have
		\begin{equation*}\label{weak_solution}
		\mathcal B_{P,B}(u, \phi):=	\int_{\Omega}\big[(a^{ij}D_ju+u\bt^i )D_i \phi+(\bb^i D_i u+cu )\phi \big] \mathrm{d}x + \int_{\Pw_{\mathrm{Rob}}}\frac{\gamma}{\beta}u \phi \,  \mathrm{d}\sigma  = 0\;
		(\text{resp.,} \geq 0),
		\end{equation*}
where $\mathrm{d}\sigma$ is the $(n-1)$-dimensional surface measure. In this case we write $(P,B)u=0$ (resp., $(P,B)u \geq 0$).  Furthermore, $u$ is a weak {\em subsolution} of \eqref{P,B}  in $\Gw$ if $-u$ is a supersolution of \eqref{P,B} in $\Gw$.
 }	
\end{defi}
The above definition should be compared with the following standard definition of weak (super)solutions in a domain $\Gw$.
\begin{defi}
	\em{
		We say that $u\in H^{1}_{\mathrm{loc}}(\Gw)$ is a \em{weak solution (resp., supersolution)} of the equation $Pu=0$ in $\Gw$ if for any (resp., nonnegative) $\phi\in C_0^{\infty}(\Gw)$ 
		$$
		\int_{\Omega}[(a^{ij}D_ju+u\bt^i )D_i \phi+(\bb^i D_i u+cu) \phi] \mathrm{d}x  = 0\;
		(\text{resp.,} \geq 0).
		$$
	}
\end{defi}
Hence, any weak solution (resp., supersolution) of the equation $(P,B)u=0$ in $\Gw$ is a weak solution (resp., supersolution)  of $Pu=0$ in $\Gw$.  In the sequel, by a (super)solution of \eqref{P,B} we always mean a weak (super)solution.  

\medskip
The {\em formal $L^2$-adjoint} of the operator $(P,B)$ is given by the operator $(P^*,B^*)$
\begin{equation}\label{P*,B*}
\begin{cases}
P^*u:=-\diver \left[A \nabla u+\bb u \right]+\bt\cdot \nabla u +cu,\\[2mm]
B^*u:=\beta \big( A \nabla u+u\bb\big)\cdot \vec{n}+\gamma u.
\end{cases}
\end{equation}
Indeed, if $A,\bt,\bb$ and $\partial \Gw$ are sufficiently smooth, then 
for any $\phi,\psi \in \mathcal{D}(\Omega,\Pw_{\mathrm{Dir}})$  satisfying  $B\psi=B^* \phi =0$ on $\Pw_{\mathrm{Rob}}$ in the classical sense, we have
\begin{align*}
&  \int_{\Omega} P(\psi) \phi \dx=
\int_{\Omega} \big(-\diver(A\nabla \psi+\bt\psi )+\bb\cdot \nabla \psi+c\psi\big)\phi\dx = 
\\
& \int_{\Pw_{\mathrm{Rob}}} \psi \phi\big(\frac{\gamma}\beta+  \bb\cdot \vec{n}\big)  \dsigma+
\int_{\Omega}\big((A\nabla \psi+\bt\psi )\cdot \nabla \phi
-\psi \nabla\cdot (\bb \phi)+c\psi\phi \big)\dx=
\\
&
\int_{\Omega} \psi\big(\!-\diver(A\nabla \phi+\bb\phi )+\bt\cdot \nabla \phi+c\phi\big)\!\dx=
\int_{\Omega}\psi P^* (\phi) \dx.
\end{align*}

\medskip

Finally, we define the notion of nonnegativity of the operator $(P,B)$.
\begin{defi}
	\em{
		We denote the cones of all positive solutions and positive supersolutions of the equation $(P,B)u=0$ in $\Gw$ by $ {\mathcal H}^0_{P,B}(\Omega)$ and $\mathcal{SH}_{P,B}(\Omega)$, respectively.
		The operator  $(P,B)$ is said to be {\em nonnegative in $\Omega$} (in short, 
		$(P,B)\geq 0$) if $\mathcal{H}^{0}_{P,B}(\Omega)\neq \emptyset$.
	}
\end{defi}
\section{Local theory}\label{sec_local}
In the present section we study the mixed value problem \eqref{P,B} in a {\em bounded} Lipschitz domain $\Gw\subset \R^n$. 
In particular, we discuss the coercivity in $H^1_{\pwd}(\Gw)$ of the bilinear form $\mathcal B_{P,B}$ associated with  the operator $(P,B)$, the validity of a weak maximum principle, and the existence of a principal eigenfunction. Unless otherwise stated:

{\bf We assume throughout this section that Assumptions~\ref{assump1}  are satisfied in $\Gw$}. 

\subsection{Coerciveness, solvability and compactness}
In this subsection we discuss the coercivity of $\mathcal B_{P,B}$ in a bounded Lipschitz domain. First, we recall the trace inequality (see for example \cite[Corollary~5.23]{L}).
\begin{lem}[Trace inequality]\label{Trace_ineq}
	Assume that  $\Gw$ is a  bounded Lipschitz domain.
Then there exists $C(\Gw)>0$ such that for any $u\in H^1(\Gw) $ and $\varepsilon>0$
		\begin{equation}\label{trace_ineq}
	\int_{\partial \Omega }u^2  \dsigma \leq
	C(\Gw) \left ( \varepsilon \| \nabla u\|_{L^2(\Omega)}^2+\frac{1}{\varepsilon} \| u\|_{L^2(\Omega)}^2\right ).
	\end{equation}
\end{lem}
The trace inequality implies the following Poincar\'e inequality:
\begin{lem}[Poincar\'e inequality]\label{poincare 1.5}
	Assume that  $\Gw$ is a  bounded Lipschitz domain.
	Then there exists $C(\Gw)>0$ such that 
	\begin{equation}\label{Poin}
	\|u \|_{H^1(\Gw)}^2 \leq C(\Gw) \left ( \| \nabla u \|^2_{L^2(\Gw)}+\int_{\Pw_{\mathrm{Rob}}} |u|^2 \!\dsigma\right ) \qquad \forall u\in H^{1}_{\pwd}(\Gw).
	\end{equation}
\end{lem}
\begin{proof}
	The  Poincar\'e inequality \cite[Lemma~5.22]{L} implies  that there exists $C(\Gw)>0$ such that 
	\begin{equation*}
	\|\phi \|_{L^2(\Gw)}^2\leq C(\Gw)\left( \| \nabla \phi \nonumber \|_{L^2(\Gw)}^2+
	\| \phi\|_{L^2(\Pw)}^2
	\right)
	\qquad  \forall \phi\in H^1(\Gw).
	\end{equation*}
	Therefore, for any $\phi \in \mathcal{D}(\Omega,\Pw_{\mathrm{Dir}})$ (which in particular, vanishes on a neighborhood of $\Pw_{\mathrm{Dir}}$), we have  
	\begin{align*}
	&
	\|\phi\|_{H^1(\Gw)}^2=\|\nabla \phi\|_{L^2(\Gw)}^2+
	\|\phi\|_{L^2(\Gw)}^2  \leq \\ &
 (1+C) \left (
\| \nabla \phi\|_{L^2(\Gw)}^2+\|\phi\|_{L^2(\Pw)}^2
 \right )= \tilde{C}\left (
 \| \nabla \phi\|_{L^2(\Gw)}^2+\|\phi\|_{L^2(\Pw_{\mathrm{Rob}})}^2
 \right ).
\end{align*}
For a general $u\in H^{1}_{\pwd}(\Gw)$, let  $\{\phi_k\}_{k\in \N}\subset \mathcal{D}(\Omega,\Pw_{\mathrm{Dir}})$ be an approximating sequence which  converges to $u$ in $H^1(\Gw)$.
Then,
\begin{align*}&
\| u\|_{H^1(\Gw)}\leq \|u-\phi_k\|_{H^1(\Gw)}+\|\phi_k\|_{H^1(\Gw)} \leq \\&
\|u-\phi_k\|_{H^1(\Gw)}+\tilde C \left (\| \nabla \phi_k\|_{L^2(\Gw)}+\|\phi_k\|_{L^2(\Pw_{\mathrm{Rob}})}\right ).
\end{align*}
The trace inequality in $H^1(\Gw)$ then implies
$$
\lim_{k \to \infty}\int_{\Pw_{\mathrm{Rob}}}{|\phi|_k^2}\text{d}\sigma=
\int_{\Pw_{\mathrm{Rob}}}{|u|^2} \text{d}\sigma,
$$
and therefore, we are done.
\end{proof}
Lemma~\ref{Trace_ineq} and Lemma~\ref{poincare 1.5} imply the coercivity of the bilinear form associated with the operator $(P+\gm,B)$ for a large enough constant $\gm$.
	\begin{theorem}[Coercivity and compactness]\label{coercivity}
		Assume that  $\Gw$ is a  bounded Lipschitz domain, and that Assumptions~\ref{assump1}  are satisfied in $\Gw$. Consider  the quadratic form
		$$
		J_{P,B}[u]\!:=\!\mathcal B_{P,B}(u, u)\!=\!\!\int_{\Omega}\![(A \nabla u+u\bt)\cdot \nabla u +(\bb\cdot \nabla u+cu)u ]\!\dx
		+\!\int_{\!\Pw_{\mathrm{Rob}}}\!\! \frac{\gamma}{\beta}|u|^2  d\sigma, \quad u\!\in\! H^{1}_{\Pw_{\mathrm{Dir}}}(\Omega).
		$$
		Then 
		\begin{enumerate}[label=\textbf{(\arabic*)}]
		\item[(i)] The bilinear form $\mathcal B_{P,B}(\cdot,\cdot)$ is bounded on $H^{1}_{\Pw_{\mathrm{Dir}}}(\Gw) \times H^{1}_{\Pw_{\mathrm{Dir}}}(\Gw)$, and there exists $\gm_0\in\R$ such that for  any $\gm> \gm_0$ the quadratic form  $J_{P,B}[u]+\gm\int_{\Omega}|u|^2 \!\dx$ is coercive on $H^1_{\Pw_{\mathrm{Dir}}}(\Omega)$. 
			\item[(ii)] For any $f\!\in \! L^2(\Gw)$ and $\gm\!>\!\gm_0$ there exists a unique function  $u_f\!=\!L_{\gm}^{-1}f\in H^{1}_{\Pw_{\mathrm{Dir}}}(\Gw)$ satisfying  
			$$\mathcal B_{P,B}(u_f, v)+\gm\int_{\Gw}u_fv \dx=\int_{\Gw}fv \dx   \qquad  \forall v\in H^{1}_{\Pw_{\mathrm{Dir}}}(\Gw).$$
			\item[(iii)] For any $\gm>\gm_0$, $L^{-1}_{\gm}:L^2(\Gw)\to L^2(\Gw)$ is a  compact linear operator.  
		\end{enumerate}
	\end{theorem}
\begin{proof} We follow the method in \cite{ST}.\\	({\em i}) Using the Sobolev embedding theorem, the trace inequality, together with Assumptions~\ref{assump1}, and Young's inequality, we obtain the following estimates for any $(u,v)\in H^{1}_{\Pw_{\mathrm{Dir}}}(\Omega)\times H^{1}_{\Pw_{\mathrm{Dir}}}(\Gw)$ and $\vge>0$:
\medskip
	\begin{enumerate}
		\item  $\left |\int_{\Omega} a^{ij}D_i uD_j v \dx \right | \leq \|a^{ij} \|_{L^{\infty}(\Gw)} \|\nabla u \|_{L^2(\Gw)}\|\nabla v \|_{L^2(\Gw)}$, \\
		\item $ \Theta^{-1} \| \nabla u \|_{L^2(\Omega)}^{2} \leq  \int_{\Omega} a^{ij}D_i uD_j u \dx \leq \Theta \| \nabla u \|_{L^2(\Omega)}^{2}$, \\
		\item $ \left |\int_{ \Omega} \bt^i u D_i v \dx \right | \leq \|\bt \|_{L^n(\Omega)}
		\| u\|_{L^{2^*}(\Omega)} \| \nabla v \|_{L^2(\Gw)}  \leq C
		\left(  \| \nabla u\|_{L^2(\Omega)}\| \nabla v\|_{L^2(\Omega)}\right)$, \\
		\item $ \left |\int_{ \Omega} \bt^i u D_i u \dx \right | \leq \|\bt \|_{L^n(\Omega)}
		\| u\|_{L^{2^*}(\Omega)} \| \nabla u \|_{L^2(\Gw)}  \leq C
		\left( \varepsilon \| \nabla u\|_{L^2(\Omega)}^2+
		\frac{1}{\varepsilon}\|u \|_{L^2(\Omega)}^2\right)$, \\
		\item $ \left |\int_{ \Omega} \bb^i u D_i v \dx \right | \leq \|\bt \|_{L^n(\Omega)}
		\| u\|_{L^{2^*}(\Omega)} \| \nabla v \|_{L^2(\Gw)}  \leq C
		\left(  \| \nabla u\|_{L^2(\Omega)}\| \nabla v\|_{L^2(\Omega)}\right)$, \\
		\item $ \left | \int_{ \Omega} \bb^i uD_i u  \dx \right |\leq  \|\bb \|_{L^n(\Omega)}
		 \|u \|_{L^{2^*}(\Omega)}\| \nabla u\|_{L^2(\Omega)} 
		 \leq C\left(\varepsilon \| \nabla u\|_{L^2(\Omega)}^2+
		\frac{1}{\varepsilon}\|u \|_{L^2(\Omega)}^2\right)$, \\
		\item $ \left |\int_\Omega c uv  \dx \right |
		\leq \|c \|_{L^{n/2}(\Omega)}
		\|u \|_{L^{2^*}(\Omega)}\|v \|_{L^{2^*}(\Omega)}\leq
		C\left( \| u\|_{H^1(\Omega)}\| v\|_{H^1(\Omega)}
		\right),
		$ \\
		\item $\left |\int_\Omega c |u|^2 \! \dx \right |\leq \|c \|_{L^{n/2}(\Omega)}
		\|u \|_{L^{2^*}(\Omega)}^2\leq
		C\left(\varepsilon \| \nabla u\|_{L^2(\Omega)}^2+
		\frac{1}{\varepsilon}\|u \|_{L^2(\Omega)}^2\right)$, \\
		\item $ \left |\int_{\Pw_{\mathrm{Rob}}} \frac{\gamma}{\beta}|u|^2  d\sigma \right |
		\leq 
		C \left ( \varepsilon \| \nabla u \|_{L^2(\Omega)}^2+\frac{1}{\varepsilon} \| u\|_{L^2(\Omega)}^2\right )$,
	\end{enumerate}
	where $C$ depends on $P,B,\Gw$ and $n$.
	Hence,  there exists $C>0$ such that
	$$|\mathcal{B}_{P,B}(u,v) |\leq C\|u \|_{H^1(\Gw)}\|v \|_{H^1(\Gw)} \qquad \forall u,v\in H^{1}_{\pwd}(\Gw).$$ 
	Moreover, the trace inequality \eqref{trace_ineq}, and the Poincar\'e   type inequality (Lemma~\ref{poincare 1.5}) imply
	$$
	\| \nabla u\|_{L^2(\Gw)}^{2} \geq \frac{1}{C}\| u\|_{H^1(\Gw)}^{2}-
	\int_{\Pw_{\mathrm{Rob}}} |u|^2  \text{d} \sigma. 	$$
Combining this with the obtained estimates (1)-(9) for the terms of $J_{P,B}$, it follows that there exists  $C=C(n,P,B,\Omega)>0$ such that
	$$
	J_{P,B}[u]
	\geq 
	\left( \Theta^{-1}-\varepsilon C\right)	\|u \|_{H^1(\Omega)}^{2} -\frac{C}{\varepsilon}\|u\|_{L^2(\Omega)}^2 \qquad \forall u\in H^{1}_{\Pw_{\mathrm{Dir}}}(\Omega).
	$$
	Hence, there exists  $\gm$ and $\delta>0$ such that
	\begin{equation}\label{eq:coerc}
	J_{P,B}[u]+\gm\int_{ \Omega}|u|^2 \!\dx \geq \delta \| u\|^{2}_{H^1(\Omega)} \qquad \forall u\in H^{1}_{\Pw_{\mathrm{Dir}}}(\Omega). 
	\end{equation}
	({\em ii}) For a given $f\in L^2(\Gw)$ the functional $f(v):=\int_{\Gw} fv \dx$ is a bounded linear functional on $H^{1}_{\Pw_{\mathrm{Dir}}}(\Gw)$. By part ({\em i}), $J_{P+\gm,B}$ is coercive on $H^{1}_{\Pw_{\mathrm{Dir}}}$. Therefore, the  Lax-Milgram theorem implies the existence and uniqueness of the required function $u=u_f$. Moreover, the mapping $f\mapsto u_f$ defines a bounded linear operator $L^{-1}_{\gm}: L^2(\Gw)\to H^{1}_{\Pw_{\mathrm{Dir}}}(\Omega)$ by $L^{-1}_{\gm}f=u_f$.  
	\\[2mm]
	({\em iii}) For any $f\in L^2(\Gw)$ and $u=L_{\gm}^{-1}f$ we have by \eqref{eq:coerc} 
	$$\delta \| u\|^{2}_{H^1(\Gw)}\leq \mathcal{B}_{P,B}(u,u)+\gm\int_{\Gw}|u|^2 \!\dx=
	(f,u)\leq \|f\|_{L^2(\Gw)}\|u\|_{H^1(\Gw)}.
	$$
	Hence,
	$$
	\| L_{\gm}^{-1}f\|_{H^1(\Gw)}\leq C \|f\|_{L^2(\Gw)}.
	$$
	The compact embedding $H^1(\Gw)\hookrightarrow L^2(\Gw)$ implies the compactness of the embedding  $H^1_{\Pw_{\mathrm{Dir}}}(\Gw)\hookrightarrow L^2(\Gw)$.  Hence, the operator $L^{-1}_{\gm}:L^2(\Gw)\to L^2(\Gw)$ is  compact.
\end{proof}
\subsection{Regularity near a Robin-portion}
Next, we discuss regularity properties of solutions near a Lipschitz-portion of $\pwr$.
\begin{lemma}[{\cite[Theorems 5.36 and  5.42]{L}}]\label{2_ineq}
	Fix $x_0\in \pwr$ and $R>0$ such that $\Gw[x_0,4R]\subset \Gw$ and $\Sigma[x_0,4R]\subset \pwr$.
	\begin{enumerate}
		\item	Let $u$ be a nonnegative weak subsolution of 
		\begin{equation}\label{PB_OmegaR}
		\begin{cases}
		Pu=0 &  \text{in } \Gw[x_0,4R], \\
		Bu = 0 &  \text{on } \Sigma[x_0,4R].
		\end{cases}
		\end{equation}
		Then for any $l>1$ there exists  
		$C=C(P,B,R,l,n)>0$ such that 
		$$\sup\limits_{\Gw[x_0,R]}u\leq C(R^{-n/l}\|u \|_{L^l(\Gw[x_0,2R])}).$$
		\item 
		Let $u$ be a nonnegative weak supersolution of \eqref{PB_OmegaR}.  
		Then  there exists  $l>1$ and   
		$C=C(P,B,R,n)>0$
		such that the following weak Harnack inequality holds   
		$$C(R^{-n/l}\|u \|_{L^l(\Gw[x_0,2R])}) \leq 
		\inf\limits_{\Gw[x_0,R]} u.
		$$
	\end{enumerate}
\end{lemma}
By combining both inequalities in Lemma~\ref{2_ineq} we obtain
\begin{cor}[Up  to the boundary Harnack inequality]\label{Harnack_up}
	Fix $x_0\in \pwr$ and $R>0$ such that $\Gw[x_0,4R]\subset \Gw$ and $\Sigma[x_0,4R]\subset \pwr$.
	Let $u$ be  a nonnegative weak solution of  \eqref{PB_OmegaR}.
	Then  there exists a positive constant $C(P,B,R,n)$ such that
	$$	\sup\limits_{\Gw[x_0,R]} u\leq C
	\inf\limits_{\Gw[x_0,R]} u.
	$$
\end{cor}
\begin{lemma} [{Up to the boundary H\"older continuity \cite[Theorem~5.45]{L}}]\label{Holder up to}
	Fix $x_0\in \pwr$ and $R>0$ such that $\Gw[x_0,4R]\subset \Gw$ and $\Sigma[x_0,4R]\subset \pwr$. Let $u$ be  a weak solution of \eqref{PB_OmegaR}, and $0<r<R$.
	Then  there exist $C(P,B,R,n)>0$ and $0<\alpha<1$ such that 
	$$
	\underset{\Gw[x_0,r]}{\mathrm{osc}} ~ u\leq C \left( \frac{r^{\alpha}}{R^{\alpha}}\sup \limits_{\Gw[x_0,R]} u\right ).
	$$
	In particular, if $u$ is bounded in $\Gw[x_0,R]$, then $u\in C^{\alpha}(\Gw[x_0,R])$.
\end{lemma}
As a corollary of Lemma~\ref{Harnack_up} and Lemma~\ref{Holder up to} we obtain the following result.
\begin{lemma}\label{1.20}
	Fix $x_0\in \pwr$ and $R>0$ such that $\Gw[x_0,4R]\subset \Gw$ and $\Sigma[x_0,4R]\subset \pwr$.
	\begin{enumerate}
		\item[(i)] Let $u$ be  a nonzero nonnegative weak solution of  \eqref{PB_OmegaR}. Then $u\in C^{\alpha}(\overline{\Gw[x_0,R]})$ and $u>0$ in $\overline{\Gw[x_0,R]}$.
		\item[(ii)]
		If $u$ is  a positive weak supersolution of  \eqref{PB_OmegaR}, then   $\inf\limits_{\Gw[x_0,R]}u>0$.
	\end{enumerate}
\end{lemma} 
\begin{proof} ({\em i}) By part (1) of Lemma~\ref{2_ineq}, $u$ is bounded in $\Gw[x_0,R]$, and consequently, 
	Corollary~\ref{Holder up to} implies that $u\in C^{\alpha}(\Gw[x_0,R])$. Hence,  Lemma~\ref{Harnack_up} implies 
	$\inf\limits_{\Gw[x_0,R]}{u}\geq \frac{1}{C}\sup\limits_{\Gw[x_0,R]}{u}>0$.
	
	({\em ii}) Follows immediately from part (2) of Lemma~\ref{2_ineq}.
\end{proof}
	In the sequel, we shall use the following terminology.
	\begin{defi}
		\em{Let  $u\! \in \! H^1_{\loc}(\wb)$, where $\Gw$ is a Lipschitz bounded domain.
			We say that $u \! \geq \! 0 $ on $\pwd$ if $u^{-} \! \in \!   H^1_{ \pwd}(\Gw)$.
	}
	\end{defi}

\subsection{Maximum principles}
The proof of the following  weak maximum  principle is hinged on the method in \cite[Theorem~5.15]{L}, though under a different setting and regularity assumptions: 
\begin{lemma}[{Weak maximum principle}]\label{weak_maximum}		Consider the problem \eqref{P,B} and suppose that Assumptions~\ref{assump1}  are satisfied in $\Gw$. Assume further that
	\begin{equation}\label{eqpos}
	\int_{\Omega}(\bt^iD_i \phi +c\phi ) \!\dx+\int_{\Pw_{\mathrm{Rob}}} \frac{\gamma}{\beta} \phi \dsigma \geq 0 \qquad \forall \,  \phi \in H^{1}_{\Pw_{\mathrm{Dir}}}(\Omega), ~\phi \geq 0.
	\end{equation} 
	Let $u\in   H^1(\Gw)$  be a   weak supersolution of \eqref{P,B} in $\Omega$ satisfying $u^{-}\in H^1_{\Pw_{\mathrm{Dir}}}(\Gw)$.
	
	Then  $u \geq 0$ in $\Omega$.  
\end{lemma}
\begin{proof}
	Assume by contradiction that  $\sup\limits_{\Gw}(u^{-})>0$.
	For any nonnegative $\phi\in H^{1}_{\Pw_{\mathrm{Dir}}}(\Gw)$
	\begin{equation*}
	\int_{\Omega}[(a^{ij}D_ju+u\bt^i )D_i \phi+(\bb^iD_i u)\phi] \dx \geq   - \int_{\Omega} cu\phi \dx-\int_{ \partial\Omega_R }\frac{\gamma}{\beta}u \phi \dsigma.
	\end{equation*}
	Therefore,
	\begin{equation*}
	\int_{\Omega}[a^{ij}D_juD_i \phi+(\bb^i-\bt^i)D_i u\phi] \dx \geq   - \int_{\Omega}(\bt^iD_i(u\phi) +cu\phi) \dx-\int_{ \partial\Omega_R }\frac{\gamma}{\beta}u \phi \dsigma.
	\end{equation*}
	Consequently, by  \eqref{eqpos}, for any nonnegative $\phi\in \Hd$ satisfying $\phi u\leq 0$ in $\Omega$, we have
	\begin{equation}\label{eq:wmp1}
	-\int_{\Omega}a^{ij}D_j u D_i \phi \dx \leq \int_{\Omega}(\bb^i-\bt^i)D_i u\phi\dx.
	\end{equation} 
	Let $k$ satisfy $-\sup\limits_\Gw(u^{-})<k<0$ and consider the function   $\phi=-\min(u-k,0)$,
	and let $F_k=\overline{\{x\in \Omega:u(x)<k \}}$ be the support of $\phi$. Since $u^-\in \Hd$,  it follows that $ \phi\in H^{1}_{\Pw_{\mathrm{Dir}}}(\Omega)$. Moreover,  
	$$ \phi u \leq 0 \mbox{ in } \Gw, \quad Du=-D\phi \mbox{ in } F_k, \mbox{ and } D\phi =0 \mbox{ in } \Omega \setminus F_k.$$  
	Consequently, the uniform ellipticity, H\"older's inequality, and \eqref{eq:wmp1} imply
	\begin{align*}
	& \Theta^{-1} \| \nabla \phi\|_{L^2(F_k)}^{2}\leq \int_{F_k}a^{ij}D_j \phi D_i \phi  \dx
	=-\int_{\Omega}a^{ij}D_j u D_i \phi \dx \leq \\ & \int_{\Omega}(\bb^i-\bt^i)D_i u\phi\dx=
	\int_{F_k}(\bb^i-\bt^i)D_i u\phi\dx
	\leq  
	\|\bb ^ { i } - \bt ^ { i }\|_{L^p(F_k)}\| \nabla u\|_{L^2(F_k)} \| \phi\|_{L^{q}(F_k)},  
	\end{align*}
	where $q=\frac{2p}{p-2}<\frac{2n}{n-2}$.
The Gagliardo-Nirenberg-Sobolev inequality  \cite[Theorem~5.8]{L}  implies that for any $N>2\geq n$
	\begin{equation*}
	\left (\int_{\Omega}|\phi|^{\frac{2N}{N-2}} \dx \right)^\frac{N-2}{N} \leq
	C\left(\int_{\Gw}|\phi|^2 \dx\right)^{\frac{N-n}{N}}
	\left(\int_{\Gw}|D\phi|^2 \dx \right)^{\frac{n}{N}}. 
	\end{equation*}
	Let $N$ be such that $\frac{2N}{N-2}>q$. Then,
	\begin{align*}
	&
	\left (\int_{\Omega}|\phi|^{\frac{2N}{N-2}} \dx \right)^\frac{N-2}{N} \leq
	C\left(\int_{\Gw}|\phi|^2 \dx\right)^{\frac{N-n}{N}}
	\left(\int_{\Gw}|D\phi|^2 \dx \right)^{\frac{n}{N}}\leq \\
	&  
	\tilde{C} \|\bb ^ { i } - \bt ^ { i }\|_{L^p(F_k)}^{n/N}
	\left(\int_{\Gw}|\phi|^2 \dx\right)^{\frac{2N-n}{2N}}
	\left(\int_{F_K}|\phi|^q \dx \right)^{\frac{n}{Nq}}.
	\end{align*}
	By H\"older inequality we obtain  for  $s=2N/(N-2)$
	\begin{align*}
	\left(\int_{\Gw}\!\!|\phi|^2 \!\dx\!\right)^{\frac{2N-n}{2N}}\!\!\!
	\left(\!\!\int_{F_k}|\phi|^q \!\dx\! \right)^{\frac{n}{Nq}} 
	\!\!\leq &
	 |\Omega|^{\frac{2N-n}{2N}\cdot \frac{s-2}{s}} 
	\left (\int_{ \Omega} |\phi|^{s}\!\dx\!\right)^{\frac{2N-n}{2N}\cdot\frac{2}{s}}
	\!\!|F_k|^{\frac{s-q}{s}\cdot\frac{n}{Nq}}\!\!
	\left( \!\int_{\Omega}\!|\phi|^{s}\!\dx\!\right)^{\frac{q}{s}\cdot\frac{n}{Nq}}\!\!=
	\\&
|\Omega|^{\frac{2N-n}{2N}\cdot \frac{s-2}{s}} 
	|F_k|^{\frac{s-q}{s}\cdot\frac{n}{Nq}}
	\left (\int_{\Omega}|\phi|^{s} \dx \right)^\frac{2}{s}.
	\end{align*} 
	If $\|\bt- \bb\|_{L^p(\Gw,\R^n)}\neq 0$, then
	 $|F_k|$ is bounded from below by a positive constant independent of $k$, and therefore,  the set $\underset{k}{\cap} F_k$ has a positive measure.
	Consequently, letting $k\to -\sup\limits_\Gw(u^{-})$, it follows that $u$ attains a finite infimum on a set  $F^*$ of positive measure.
	On $F^*$ we have $Du=0$, however, $F^*$ contains $\underset{k}{\cap} F_k$ and  $Du\neq 0$  on  $\underset{k}{\cap} F_k$, a contradiction.

	On the other hand, if $\|\bt- \bb\|_{L^p(\Gw,\R^n)}= 0$, it follows that  $\phi=0$, and then $u\geq k$. Letting $k\to 0$ implies $u\geq 0$.
\end{proof}
	\begin{rem}
		\em{ 
			The assumption in Lemma~\ref{weak_maximum} that $\pwd$ contains a nonempty Lipschitz-portion of $\Pw$  is essential. Indeed, if $\pwd=\emptyset$, then the lemma's assumptions imply that either $u=\mathrm{constant}$ or else   $u \geq 0$, see \cite[Theorem~5.15]{L}.
		}
	\end{rem}We recall the following (interior) weak Harnack inequality.
\begin{lem} [{Weak Harnack inequality \cite[Thereom~3.13]{MZ}}]\label{lem_whi}
	Let Assumptions~\ref{assump2} hold in a domain $\Gw \subset \R^n$, and 
	let $u$ be a weak supersolution of the equation $Pu  = 0$ in $\Gw$.
	Assume that $u \geq 0 $  in some open ball $B(r) \subset \Gw$ of radius $r>0$.
	Then for any $\varrho, \tau \in (0,1)$ and $\gamma \in (0,n/(n-2))$, there is a positive constant $C>0$ depending on $P, n,\gg,\varrho,\gt$ and $r$,
	such that
	\begin{equation}
	\left(\frac{1}{|B(\varrho r)|}\int_{B(\varrho r)} u^{\gamma} \dx \right )^{1/\gamma} \leq
	C
	\inf_{B(\tau r)} u.
	\end{equation}
\end{lem}
Lemma~\ref{2_ineq} and Lemma~\ref{lem_whi}) imply:
\begin{lem}[Strong maximum principle]\label{maxagmon}
	Let Assumptions~\ref{assump2} hold in a domain $\Gw \subset \R^n$,  and let $u$ be a nonnegative supersolution of \eqref{P,B}  in $\Gw$. Then either $u$ is strictly positive in $\Gw\cup \pwr$, or else, $u=0$ in $\Gw$.
\end{lem}
\begin{defi}{\em 
	We say that the {\em generalized maximum principle} holds in a bounded domain $\Gw$ if for any $u\in H^1(\Gw)$ satisfying 
	$(P,B)u\geq 0$ in $\Gw$ and $u^{-}\in H^1_{\pwd}(\Gw)$,  we have
	$u \geq 0$ in $\Gw$.
}
\end{defi}
We recall the notion of  the ground state transform which implies a generalized maximum principle that holds when assumption \eqref{eqpos} is replaced by 
\begin{equation*}\label{SH_bdd} 
 \mathcal{RSH}_{P,B}(\Omega)\! :=\!\Big\{\!u\!\in\! \mathcal{SH}_{P,B}(\Omega)\mid u, u^{-1}\!, uPu \!\in\! L^\infty_\loc(\bar \Gw\setminus \pwd) \mbox{ and }  \frac{uBu}{\gb}\!\in\!  L^\infty_\loc(\Pw_{\mathrm{Rob}})  \!\Big\} \!\neq\!\emptyset.
\end{equation*}
\begin{defi}{\em  $u\in \mathcal{RSH}_{P,B}(\Omega)$ is called a {\em regular positive supersolution}. Note that any $u\in \mathcal{H}^{0}_{P,B}(\Gw)$ is a regular positive supersolution of \eqref{P,B} in $\Gw$.
	}
\end{defi}
\begin{defi}[Ground state transform]\label{gs+transform} {\em  Assume that  $u\in \mathcal{RSH}_{P,B}(\Omega)$, and consider the   bilinear  form 
		$$
		\mathcal B_{P^u,B^u}(\phi,\psi):=\mathcal B_{P,B}(u\phi,u\psi),
		$$
		where  $\phi u ,\psi u \in \mathcal{D}(\Omega,\Pw_{\mathrm{Dir}})$.
		We note that the form $\mathcal B_{P^u,B^u}$ corresponds to the elliptic operator $(P^u,B^u)$, where 
		$$
		P^u(w):=\frac{1}{u}P(uw),\qquad  \mbox{with the boundary operator } B^u(w):=\frac{1}{u}B(uw).
		$$
		 The operator $(P^u,B^u)$ is called the {\em ground state transform} of $(P,B)$ with respect to $u$. The operators  $P^u(w)$ and $B^u(w)$ are given explicitly by 
		\begin{equation}\label{eq-gs}
		P^u(w)\!=\!-\frac{1}{u^2}\diver\!(u^2A(x)\nabla w)\!+\!\left[	\bb-\bt \right ]\!\nabla w+w\frac{Pu}{u}\, , \mbox{ and }
		B^u(w)\!=\!\beta  A\nabla w \cdot\vec{n}+ w\frac{Bu}{u}\,.
		\end{equation}
		We say that $w$ is a {\em weak (resp.,  regular  super)solution of} $(P^u,B^u)$ in $\Omega$,  if  $w\in H^{1}_{\loc}(\wb)$  
		and for any (resp., nonnegative) $\phi  \in  \mathcal{D}(\Omega,\Pw_{\mathrm{Dir}})$ we have 
		\begin{multline}\label{eq:gs_eq}
		\mathcal B_{P^u,B^u}(w,\phi)\!=\!\int_{\Omega}\!\!\big[a^{ij}D_iw D_j\phi + (\bb^i - \bt^i)D_i w\phi \big] u^2\!\dx +  \\
		\int_{\Omega}\!\!  w\phi (Pu)u\dx  + 
		\!\int_{\Pw_{\mathrm{Rob}}}\!\!\!\!\! w \phi   \frac{u Bu}{\gb}\!\dsigma
	 = 0 \;(\text{resp., } \geq 0). 
		\end{multline}
	}
\end{defi}
Note that  $\mathcal{B}_{P^u,B^u}$ is defined on   $L^2(\Gw, u^2 \!\dx)$. 
\begin{rem}\label{r2.11}
	 \em{
		\begin{enumerate}
			\item 	Let  $u,v\in \mathcal{H}_{P,B}^{0}(\Omega)$ (resp., $u\in \mathcal{H}_{P,B}^{0}(\Omega),v\in \mathcal{SH}_{P,B}(\Omega)$). Assume that $v/u \in H^{1}_{\mathrm{loc}}(\overline{\Omega}\setminus \pwd, u^2 \!\dx)$.
			Then $v/u$ is a weak positive (resp., super)solution of the equation $(P^u,B^u)w=0$ in $\Gw$.
			\item  If $u\in \mathcal{H}^{0}_{P,B}(\Omega)$, then for  any $w \in H^{1}_{\loc}(\overline{\Omega}\setminus \pwd)$ and $\phi \in  \mathcal{D}(\Omega,\Pw_{\mathrm{Dir}})$  we have 
		$$
			\mathcal{B}_{P^u,B^u}(w,\phi)=\int_{\Omega}\big[a^{ij}D_iw D_j\phi + (\bb^i - \bt^i)D_iw \phi \big]u^2\!\dx =\mathcal{B}_{P,B}(uw,  u\phi ). 
		$$  
		
		\end{enumerate}		
	}
\end{rem}
Consequently, we have:
\begin{cor}\label{rem_reg}
	Let Assumptions~\ref{assump1} hold  in a bounded Lipschitz domain $\Gw$, and let $u\in \mathcal{RSH}_{P,B}(\Gw)$. Then   $u>0$ on $\wb$. Moreover, the weak maximum principle  holds for the operator $(P^u,B^u)$ in any Lipschitz domain  $\Gw'\Subset_R \Gw$.
	
	Furthermore, if $u\in \mathcal{H}^{0}_{P,B}(\Gw)$ and $u, u^{-1}\in L^\infty(\Gw)$, then $u\in C^{\alpha}(\wb)$, and the weak maximum principle  holds for the operator $(P^u,B^u)$ in $\Gw$.  
\end{cor}
\begin{rem}\label{rBR}
	\em{For the relationship between the validity of the generalized maximum principle and the existence of a positive solution such that $u$ and $u^{-1}$ are bounded, see \cite{BR} for the case $\pwr=\emptyset$, and  \cite{ATG} for the case $\pwr=\emptyset$. 		
	}
\end{rem}
Corollary~\ref{rem_reg} implies the following generalized maximum principle.
\begin{lemma}[Generalized maximum principle]\label{gen_max_weak}
Let Assumptions~\ref{assump2} hold  in a domain $\Gw$, and let $u\in \mathcal{RSH}_{P,B}(\Gw)$.
	  If $\Gw'\Subset_R \Gw$  is a Lipschitz subdomain, then the generalized maximum principle  holds for $(P,B)$ in $\Gw'$. 
Moreover, if in addition,  Assumptions~\ref{assump1} hold  in a bounded Lipschitz  domain $\Gw$, $u\in \mathcal{H}^{0}_{P,B}(\Gw)$,  and  $u, u^{-1}\in   L^\infty(\bar\Gw\setminus \pwd)$. Then, the generalized maximum principle  holds for the operator $(P,B)$ in $\Gw$.  
\end{lemma}
\begin{proof}
	Apply the  ground state  transform $(P^u,B^u)$ (see Definition~\ref{gs+transform}). By Corollary~\ref{rem_reg}, the weak maximum principle  holds for $(P^u,B^u)$ in $\Gw'$. Consequently, if  $v$ is a  supersolution of $(P,B)$ in $\Gw'$ with $v^{-}\in H^1_{\partial \Gw'_{\mathrm{Dir}} }\!(\Gw')$, then $v/u$ is a supersolution of $(P^u,B^u)$ in $\Gw'$ with $(v/u)^{-}\in H^1_{\partial \Gw'_{\mathrm{Dir}} }\!(\Gw')$, and by Corollary~\ref{rem_reg}, $v/u\geq 0$ in $\Gw'$. Thus, $v\geq 0$ in $\Gw'$.
	
	 The second statement of the lemma follows from the second part of Corollary~\ref{rem_reg}. 
	\end{proof}
Next, we present a priori interior estimates for positive supersolution.
\begin{lemma}\label{bound grad} 
		Let  Assumptions~\ref{assump2} hold in a Lipschitz domain $\Gw$, and $f\in L^2_\loc(\wb)$. Assume that $v$ is a (continuous) 
		 solution of the equation $Pu=f$  in a domain $\Gw\subset \R^n$ and satisfies $Bv\!=\!0$ on $\Pw_{\mathrm{Rob}}$ (in short, $(P,B)v\!=\!f$ in $\Gw$). Then for any  Lipschitz subdomains $\omega\Subset_R \omega'  \Subset_R \Gw$, there exists a constant $C>0$ independent of $v$, such that 
		$$\|v \|_{H^1(\omega)}\leq C\left (\|v \|_{L^{\infty}(\omega')}+ \|f\|_{L^{2}(\omega')} \right).$$
\end{lemma}
\begin{proof}
	The definition of $v$ being a   solution of  $(P,B)u\!=\!f$  in $\Omega$ reads as
	\begin{equation}\label{weak_solution2}
		\int_{\Omega}[(a^{ij}D_jv+v\bt^i )D_i \phi+(\bb^i D_i v)\phi+cv \phi] \mathrm{d}x + \!\!\int\limits_{\pwr}\!\!\frac{\gamma}{\beta}v \phi   \,\mathrm{d}\sigma  \!= \!
			\int_{\Omega}\!\!f\phi \dx\quad \forall \phi\in \mathcal{D}(\Omega,\Pw_{\mathrm{Dir}}).
	\end{equation}
	
For any such $\phi$, we may pick $v\phi^2$ as a test function in \eqref{weak_solution2} to get
\begin{align*}
&
\int_\Omega a^{ij}D_jv D_iv \phi^2 \dx=
-2\int_{\Omega} [a^{ij}D_i v D_j \phi  \phi v + v^2 \phi\bt^i D_i\phi]\dx -\int_{\Omega}
\phi ^2v\bt^i \D_i v \dx
  - \\ &
\int_{\Omega} [\bb^i  D_i  v +cv]v\phi^2 \dx -
\int\limits_{\pwr}(\gamma/ \beta) v^2\phi^2 \,\mathrm{d}\sigma+\int_{\Omega}fv\phi^2 \dx.
\end{align*}
Let $\omega \Subset_R \omega'\Subset_R \Gw$ be  Lipschitz subdomains of $\Gw$, and let  $\phi\in \mathcal{D}(\omega',\partial \omega'_{\Dir} )$  satisfy 
$$
 0\leq \phi \leq 1, \quad 
\phi=1 \quad \mbox{in } \omega, \quad  |\nabla \phi|\leq \frac{1}{\mathrm{dist}(\partial\omega_{\mathrm{Dir}},\partial\omega'_{\mathrm{Dir}})} \quad \mbox{in } \omega'.
$$
Then one obtains as in the proof of Theorem~\ref{coercivity} that
\begin{multline*}
\Theta_{\omega'}^{-1} \|\nabla v \|_{L^2(\omega)}^2 \leq \int_{\omega'} a^{ij}D_jv D_iv \phi^2
\dx= \\ 
-2\int_{\omega'}[a^{ij}D_i v D_j \phi  \phi v  +v^2 \phi\bt^i D_i\phi]\dx-\int_{\omega'}
\phi ^2v\bt^i \D_i v \dx -\\ 
\int_{\omega'} (\bb^i  D_i  v +cv)v\phi^2 \dx -
\int_{\partial \omega'_{\mathrm{Rob}}}(\gamma/ \beta) v^2\phi^2 \mathrm{d}\sigma +\int_{\gw'}fv\phi^2 \dx\leq \\ 
C_1(P,B,\omega,\gw')\left ( \varepsilon \|\nabla v \|_{L^2(\omega')}^2+\frac{1}{\varepsilon}\|v \|_{L^2(\omega')}^2\right )
+ \varepsilon \|v\|_{L^2(\omega')}+\frac{1}{\varepsilon}\| f\|_{L^2(\omega')}.
\end{multline*}
Therefore, for $\varepsilon>0$ sufficiently small there exists $C=C(P,B,\omega,\gw',\vge)$ such that
$$\|v\|_{H^1(\omega)}\leq
	C\|v\|_{L^2(\omega')} +
	\frac{1}{\varepsilon}\| f\|_{L^2(\omega')}
\leq
	C|\gw'|\|v\|_{L^{\infty}(\omega')}+
	\frac{1}{\varepsilon}\| f\|_{L^2(\omega')}.\qedhere$$
\end{proof}
\subsection{The spectrum of $(P,B)$}
Next, we discuss spectral properties of the operator $(P,B)$ in a bounded Lipschitz domain $\Gw$  satisfying Assumptions~\ref{assump1}.
\begin{defi}
	\em{		Let $(\tilde{P},\tilde{B})$ be the realization of the operator $(P,B)$ in $\Gw$ with the domain 
		$$D(\tilde{P},\tilde{B}):=\{ u\in H^{1}_{\pwd}(\Gw)\subset L^2(\Gw) \mid (P,B)u=f\in L^2(\Gw)\},$$
		and let $\gs (\tilde{P},\tilde{B})$ and $\gr(\tilde{P},\tilde{B})$ be the spectrum and the resolvent set of $(\tilde{P},\tilde{B})$ in $L^2(\Gw)$, respectively.
		We define the {\em `bottom\!' of the spectrum of $(\tilde{P},\tilde{B})$ in $\Gw$} by 
		$$
		\Gamma=\Gamma(P,B,\Gw):=\inf\{ \mathrm{Re}(\lambda)\mid \lambda\in \gs(\tilde{P},\tilde{B})\}.
		$$		
		For $\gl\in \gr(\tilde{P},\tilde{B})$, we denote the resolvent operator of $(\tilde{P}-\gl ,\tilde{B})$ in $L^2(\Gw)$ by $R_\Gw(\gl)$.
		Set  $$
		\Gl= \Gl(P,B,\Gw) :=\sup\{\lambda:\mathcal{B}_{P-\lambda,B}(\phi,\phi)\geq 0, ~ \forall \phi\in \mathcal{D}(\Gw,\pwd) \}.
		$$
		In the sequel we might omit the dependence on $P,B,\Gw$ when there is no danger of ambiguity.
	}
\end{defi}
\begin{lem}\label{lem: closed operator}
	 $(\tilde{P},\tilde{B}):D(\tilde{P},\tilde{B})\to L^2(\Gw)$ is a closed operator.
	
\end{lem}
\begin{proof}
Let $\{u_k\}_{k\in \N}\subset D(\tilde{P},\tilde{B})$
	satisfying 
	$$ \lim\limits_{k\to \infty}\|u_k-u\|_{L^2(\Gw)}=0,\quad 
	(\tilde{P},\tilde{B})u_k=f_k\in L^2(\Gw), \quad  \mbox{and } \lim\limits_{k\to \infty}\|f_k-f\|_{L^2(\Gw)}=0.$$
	Then $(P+\mu,B)u_k=\mu u_k+f_k$ and thus $u_k=(P+\mu,B)^{-1}(\mu u_k+f_k)$.
	By Theorem~\ref{coercivity} the operator $(P+\mu,B)^{-1}:L^2(\Gw)\to H^{1}_{\pwd}(\Gw)$ is linear and bounded. By letting $k\to \infty$ we obtain that $u=(P+\mu,B)^{-1}(\mu u+f)\in H^{1}_{\pwd}(\Gw)$. \\  Hence, $u\in D(\tilde{P},\tilde{B})$ and $(P,B)u=f$. 
\end{proof}
\begin{rem}\label{rem-disc}
	\em{
		Let Assumptions~\ref{assump1} hold  in a domain $\Gw$.
		Recall that by Theorem~\ref{coercivity}, the resolvent operator $R_\Gw(\gl)$ is compact on $L^2(\Gw)$ for large enough $\gl\in \R$. 
		Therefore, $R_{\Gw}(\gl)$  is compact for any $\lambda\in \gr (\tilde{P},\tilde{B})$ \cite[Proposition~2.7.6]{Lablee}. In particular, for any $\lambda\in \gr(\tilde{P},\tilde{B})$, and $ f\in C^{\infty}_0(\Gw)$, the function  $R_{\Gw}(\lambda)f$ is a  solution to the problem 
		\begin{align*}	
		(P-\lambda) v=f& \quad \text{~in~} \Gw, \\
		Bv=0 &\quad \text{~on~} \pwr, \mbox{ and } v\in H^1_{\Pw_{\mathrm{Dir}}}(\Gw).
		\end{align*}
		Moreover, Lemma \ref{lem: closed operator} and the compactness of $R_{\Gw}(\gl)$ implies that $\sigma(\tilde{P},\tilde{B})=\sigma_{\mathrm{point}}(\tilde{P},\tilde{B})$ (see for example, \cite[Theorem~2.7.8]{Lablee}). 
		In Theorem~\ref{agmonmethod} we show that in fact,
	$\Gamma\in \sigma_{\mathrm{point}}(\tilde{P},\tilde{B})$.
	}
\end{rem}
\begin{cor}\label{cor_glc}
	Fix $\lambda\in \gr(\tilde{P},\tilde{B})$. Then $\lambda'$ in an eigenvalue of $(\tilde{P},\tilde{B})$ if and only if $\lambda'=\gm^{-1} +\lambda$, where $\mu\neq 0$ is an eigenvalue of $R_{\Gw}(\lambda)$. In particular, the spectrum of $(\tilde{P},\tilde{B})$ consists of only  isolated eigenvalues  of finite multiplicity.
\end{cor}
\begin{proof}
	By Remark~\ref{rem-disc}, $\sigma(\tilde{P},\tilde{B})=\sigma_{\mathrm{point}}(\tilde{P},\tilde{B})$. Hence, by the definition of $R_{\Gw}(\gl)$, we see that $\mu\in \gs(R_{\Gw}(\gl))$ and $\gm\neq0$ if and only if   $\gl':=\gm^{-1}+\gl \in \sigma_{\mathrm{point}}(\tilde{P},\tilde{B})$.
\end{proof}
\subsection{Exhaustion and Harnack convergence principle}
We proceed with Harnack convergence principle for a sequence of positive solutions in a general domain $\Gw$. First, we define a (Lipschitz) exhaustion of $\wb$.
\begin{defi}\label{def:exhaustion}
	\em{
		A sequence $\{\Omega_k\}_{k\in \mathbb{N}}$ 
		is called an {\em exhaustion}  of $\wb$ if it is
		an increasing sequence of Lipschitz bounded domains such that ${\Omega_k}\Subset_R \Omega_{k+1} \Subset_R  \Gw$,
		and $\bigcup\limits_{k\in \N} \overline{\Omega_k}=\overline{\Gw}\setminus \Pw_{\mathrm{Dir}}$. 
		For $k\geq 1$ we denote : 
		$$  \partial \Gw_{k,\Rob}:=\mathrm{int}( \partial \Gw_k\cap \pwr) \quad 
		\partial \Gw_{k,\Dir}:=\partial \Gw_k \setminus \partial \Gw_{k,\Rob}, \quad 
 \Omega_{k}^{*}:=(\wb)\setminus \overline{\Omega_k}.$$
	}
\end{defi}
\begin{rem}
	\em{
		In Appendix \ref{appendix1}, we show that such an exhaustion exists once we impose the following stronger regularity assumption on $\pwr$:
		for each $x_0\in \Pw_{\mathrm{Rob}}$ there exists $R>0$ such that $\Sigma[x_0,R]$ is a $C^1$-portion of $\pwr$. This extra regularity assumption is needed to ensure that
			$\partial \Omega_k$ meets $\pwr$ in `good directions' (see Definition~\ref{def_gd}) which implies that $\Gw_k$ is indeed a Lipschitz domain. 
	}
\end{rem}
\begin{lem}[Harnack convergence principle]\label{HCP}
	Suppose that Assumptions~\ref{assump2} hold in $\Gw$, and  let $\{\Gw_k\}_{k\in \N}$ be an exhaustion of $\wb$. Let $x_0\in \Gw$ be a fixed reference point. 
	For each $k\geq 1$, let $ u_k\in H^{1}_{\loc}(\Gw_k \cup\partial \Gw_{k, \pwr})$ be a positive solution of the problem
	\begin{equation}\label{PB_Gwk}
	\begin{cases}
	Pu=0 &  \text{in } \Gw_k, \\
	Bu = 0 &  \text{on } \partial \Gw_{k,\Rob},
	\end{cases}
	\end{equation}
	satisfying  $u_k(x_0)=1$. Then the sequence $\{u_k\}_{k\in \N}$ admits a subsequence converging locally uniformly in $\wb$ to a positive solution 
	$u\in \mathcal{H}^{0}_{P,B}(\Gw)$. 
	
	Moreover, the same conclusion holds if in \eqref{PB_Gwk} one replaces $P$ by $P+V_k$, with $V_k \to 0$ in $L^{p/2}_\loc(\Gw)$, where $p>n$. 
\end{lem}
\begin{proof}
	Fix $k\in \N$. By  the local Harnack inequality \cite[Theorem~8.20.]{GT}, and the up to the boundary Harnack inequality (Corollary~\ref{Harnack_up}),   the sequence $\{u_{j}\}_{j >k}$ is locally uniformly bounded in $\overline{\Gw_k}$.
	Moreover, Lemma~\ref{Holder up to} implies that 
	this sequence is bounded in $C^{\alpha}(\overline{\Gw_{k-1}})$. Hence, by Arzel\'a-Ascoli theorem 
	$\{u_j\}_{j\in \N}$ admits a subsequence $\{u_{j_l}\}_{l\in \N}$ converging uniformly to $\tilde u_{k}\in  C^{\alpha}(\overline{\Gw_{k-1}})$. Moreover, by Lemma~\ref{bound grad}, $\| \nabla u_{j_l}\|_{L^2(\Gw_{k-1})}$ is uniformly bounded.
	In fact, in light of Lemma~\ref{bound grad}, by extracting a subsequence,  we may assume  that $\{u_{j_l}\}_{j\in \N}$ also converges weakly in $H^1(\Gw_k)$ and almost everywhere to the positive function $\tilde u_k$.
	In particular, for each $1\leq i\leq n$, we clearly have that  $D_i \tilde u_k$ is a linear functional on $L^2(\Gw_k)$ satisfying for all $\phi\in C^{\infty}_0(\Gw_k)$
	\begin{equation}\label{nabnab}
	\int_{\Gw_k}D_i \tilde u_k \phi \dx=\lim\limits_{l \to \infty}
	\int_{\Gw_k}D_i u_{j_l} \phi\dx =-\lim\limits_{l \to \infty}\int_{\Gw_k} u_{j_l} D_i \phi \dx=-\int_{\Gw_k}\tilde u_k D_i \phi \dx.
	\end{equation}
	By Riesz representation theorem,  $D_i \tilde u_k\in L^2(\Gw_k)$ and by  \eqref{nabnab} $\nabla \tilde u_k$ is indeed the weak gradient  of $\tilde u_k$.
	Therefore, we may apply the estimates in Theorem~\ref{coercivity} to deduce that for any $\phi\in \mathcal{D}(\Omega_k,\Pw_{k,\mathrm{Dir}})$,
	$\lim\limits_{l\to \infty }\mathcal{B}_{P,B}(u_{j_l},\phi)=\mathcal{B}_{P,B}(\tilde u_k,\phi)$.
	Hence, $\tilde u_k>0$ is a positive solution of \eqref{PB_Gwk}. 
	
	Since $k$ was arbitrary, we may use the Cantor diagonal argument to extract a subsequence $\{u_{j,j}\}_{j\in \N}$ of  $\{u_j\}_{j\in \N}$  converging locally uniformly in  $\overline{\Gw}\setminus \Pw_{\mathrm{Dir}}$ to a function $u\in C^\ga(\overline{\Gw}\setminus \Pw_{\mathrm{Dir}})\cap H^{1}_{\loc}(\wb)$ which  is a positive solution of  the problem \eqref{P,B} in $\Gw$.
\end{proof}
\subsection{The generalized principal eigenvalue} Next, we introduce the notion of the generalized principal eigenvalue of $(P,B)$, and study its relation to $\Gamma$, the  `bottom\!' of $\gs (\tilde{P},\tilde{B})$.
\begin{defi}\label{def_princ}
	{\em 	
		Let $\Gw$ be a domain in $\R^n$, and let $0\lneqq V\in L^{p/2}_\loc(\wb)$, where $p>n$. The {\em generalized principal eigenvalue} of $(P,B)$ in $\Gw$ with respect to $V$  is defined by 
		$$
		\lambda_0=\lambda_0(P,B,V,\Gw):=\sup\{ \lambda\in \R \mid \mathcal{H}^{0}_{P-\gl V,B}(\Gw)  \neq \emptyset \} .
		$$
Unless otherwise stated,  we always assume that $V=1$ in the definition of $\lambda_0$, and we usually omit the dependence on $P,B,V$ and $\Gw$.
	}
\end{defi}
The Harnack convergence principle (Lemma~\ref{HCP}) implies: 
\begin{cor}\label{cor_gwmp}
	$$
	\lambda_0(P,B,V,\Gw):=\max\{ \lambda\in \R \mid \mathcal{H}^{0}_{P-\gl V,B}(\Gw)\neq \emptyset \}, \mbox{ i.e. } \mathcal{H}^{0}_{P-\gl_0 V,B}(\Gw)\neq \emptyset.
	$$
\end{cor}
\begin{lem}[{Generalized maximum principle, cf. \cite[Theorem~2.6]{Agmon}}]\label{324}
	Let $\Gw$ be a domain in $\R^n$. For any $\gl \leq  \lambda_0(P,B,1,\Gw)$,  the operator $(P-\gl,B)$ satisfies the generalized maximum principle in any Lipschitz  subdomain  $\Gw'\Subset_R \Gw$. Hence, $(-\infty,\lambda_0(P,B,1,\Gw)) \subset  \gr(\tilde P,\tilde B,\Gw')$.
\end{lem}
\begin{proof}
By Corollary~\ref{cor_gwmp},  there exists $u\in \mathcal{H}^{0}_{P-\gl_0,B}(\Gw)$.   Consequently, for any  Lipschitz  subdomain  $\Gw'\Subset_R \Gw$, we have $u>0$ in $\overline{\Gw'}$. In light of Lemma~\ref{gen_max_weak}, $(P-\gl,B)$ satisfies the generalized maximum principle  in $\Gw'$ for any $\gl\leq \gl_0$.
	
	 The generalized maximum principle implies the uniqueness of a solution $v$ to the problem 
	\begin{align*}	
	(P-\lambda) v&=f\in L^2(\Gw')  \quad \text{~in~} \Gw', \\
	Bv&=0 \quad  \text{~on~} \pwr',\mbox{ and } v\in H^1_{\partial\Gw'_{\mathrm{Dir}}}(\Gw').
	\end{align*}  
Therefore, the Fredholm alternative implies that $\gl\not\in \gs(\tilde{P},\tilde{B},\Gw')$.
\end{proof}
The following lemma is well known in the case $\pwr=\emptyset$ \cite[Lemma~2.7]{Agmon}.
\begin{lem}\label{u_minus}
	Let $\Gw$ be a Lipschitz domain, and suppose that  $u\in  H^{1}_{\loc}(\wb)$ satisfies $(P-\lambda,B)u \geq 0$ in $\Gw$ for some $\gl\in \R$. Then $(P-\lambda,B)u^{-} \leq 0$ in $\Gw$.
\end{lem}
\begin{proof}
	Without loss of generality, we may assume that $\gl=0$.
	Clearly $u^{-}$ and $|u|$ belong to $ H^{1}_{\loc}(\wb)$.
	For any $\varepsilon>0$ define $u_{\varepsilon}=\sqrt{u^2+\varepsilon^2}$.
	Then $u_{\varepsilon},u/u_{\varepsilon}$ belong to $H^{1}_{\loc}(\wb)$ and $u_{\varepsilon}\to |u|$ in $H^{1}_{\loc}(\wb)$ as $\varepsilon \to 0$.
	Next, let $0\leq \phi\in \mathcal{D}(\Gw,\pwd)$. Then 
	\begin{multline*}
	A(x)\nabla u_\vge \cdot \nabla \phi=
	A(x)\frac{u}{u_{\varepsilon}}\nabla u \cdot\nabla \phi   
	\leq 
	\\
	A(x)\nabla u\cdot \frac{u}{u_{\varepsilon}} \nabla \phi +A(x)\nabla u\cdot\frac{\phi} {u_{\varepsilon}}\left(1-\frac{u^2}{u_{\varepsilon}^2}\right)\nabla u=
	A(x)\nabla u\cdot \nabla \left( \phi\frac{u}{u_{\varepsilon}}\right )
	\quad \text{~a.e~ in~} \Gw.
	\end{multline*}
Consider the function   $\phi_{\varepsilon}:=\frac{1}{2}(1-\frac{u}{u_{\varepsilon}})\phi$, then
	$$
	A(x)\frac{\nabla (u_{\varepsilon}-u)}{2}\cdot \nabla \phi \leq 
	-A(x)\nabla u \cdot \nabla \phi_{\varepsilon.}.
	$$
	Note that $0\leq \phi_{\varepsilon}\in H^{1}_{\pwd}(\Gw)\cap L^{\infty}(\Gw)$. 
	Since $(P-\lambda,B)u \geq 0$ it follows that 
	$$
	-\int_{\Gw}A(x)\nabla u \cdot \nabla \phi_{\varepsilon} \dx \leq 
	\int _{\Gw}(u\bt\cdot  \nabla \phi_{\varepsilon}+\bb \cdot \nabla u \phi_{{\varepsilon}}+cu \phi_{\varepsilon}) \dx+\int_{\pwr}
	\frac{\gamma}{\beta}u\phi_{\varepsilon}
	\text{d}\sigma.
	$$
 Therefore,
	\begin{equation}\label{agepsilon}
	\int_{\Gw}A(x) \frac{\nabla(u_{\varepsilon}-u)}{2}\cdot \nabla \phi  \dx\leq 
	\int _{\Gw}(u\bt\cdot \nabla \phi_{\varepsilon}+\bb\cdot \nabla u \phi_{{\varepsilon}}+cu \phi_{\varepsilon}) \dx+\int_{\pwr}
	\frac{\gamma}{\beta}u\phi_{\varepsilon}
	 \,\text{d}\sigma.
	\end{equation}
	Notice that $\phi_{\varepsilon} \to (\sign u^{-})\phi$ as $\varepsilon \to 0$ and  $0\leq\phi_{\varepsilon} \leq \phi$.
	Moreover, 
	\begin{align*}&
	\lim_{\varepsilon \to 0}\nabla \phi_{\varepsilon}= \lim_{\varepsilon \to 0}\frac{1}{2}\left(1-\frac{u}{u_{\varepsilon}}\right)\nabla \phi-\frac{\phi}{2}
	\lim_{\varepsilon \to 0}\frac{u_{\varepsilon}\nabla u-\dfrac{u^2\nabla u}{u_{\varepsilon}}}{u_{\varepsilon}^2}= (\sign u^{-})\nabla \phi .
	\end{align*}
	Letting $\varepsilon\to 0$ in \eqref{agepsilon}, we see that
	$$
	\int_{\Gw}A(x)\nabla u^{-}\cdot \nabla \phi  \dx\leq 
	-\int _{\Gw}(u^{-}\bt\cdot \nabla \phi+\bb \cdot\nabla u^{-} \phi+cu^{-} \phi )\dx - \int_{\pwr}
	\frac{\gamma}{\beta}u^{-}\phi
	\,\text{d}\sigma. \;\; \qedhere
	$$
\end{proof}
We need the following Kato-type inequality for the operator $(P,B)$.
\begin{lemma}[{cf. \cite[Lemma~2.8]{Agmon}}]\label{kato}
		Let $u\in  H^{1}_{\loc}(\wb;\C)$ satisfy 
		$(P,B)u=f$ in $\Gw$, where $f\in  L^1_{\loc}(\wb)$.
		Then 
		\begin{equation}\label{Kato_equa}
		(P,B)|u|\leq \mathrm{Re}\left (\frac{\overline{u}}{|u|}f\right ).
		\end{equation}
\end{lemma}
\begin{proof}
	Assume first that $\bb=\bt=c=0$.
	Let $u_{\varepsilon}:=\sqrt{|u|^2+\varepsilon^2}$ and let $0\lneqq \phi\in \mathcal{D}(\Gw,\pwd)$.
	Then $\nabla u_{\varepsilon}=\dfrac{\nabla u \overline{u}+u\nabla \overline{u}}{2u_{\varepsilon}}$. Moreover,  (cf. \cite[(5.6)]{KATO}),
	$|\nabla u|_A^2-|\nabla u_{\varepsilon}|_A^2\geq 0$ a.e in $\Gw$. 
	Thus,
	\begin{multline*}
	\int_{\Gw}A\nabla u_{\varepsilon}\! \cdot \! \nabla \phi \dx+
	\int_{\pwr}\frac{\gamma}{\beta}\phi u\frac{\overline{u}}{2u_{\varepsilon}} \!\dsigma +
	\int_{\pwr}\frac{\gamma}{\beta}\phi \overline{u}\frac{u}{2u_{\varepsilon}} \!\dsigma=\\[2mm] 
	\int_{\Gw}\!\left(\!A(x)\nabla u \!\cdot\! \nabla \!\left( \phi \frac{\overline{u}}{2u_{\varepsilon}}\!\right )\!+
	A(x)\nabla \overline{u}\!\cdot\! \nabla \left ( \phi \frac{u}{2u_{\varepsilon}}\!\right ) \!\right)\!\!\dx+
\int_{\pwr}\!\frac{\gamma}{\beta}\phi u\frac{\overline{u}}{2u_{\varepsilon}} \,\text{d} \sigma+
\int_{\pwr}\frac{\gamma}{\beta}\phi \overline{u}\frac{u}{2u_{\varepsilon}} \,\text{d} \sigma \\[2mm]
-\frac{1}{2}\int_{\Gw}\!\left(\!	A\nabla u\cdot \nabla \left ( \frac{\overline{u}}{u_{\varepsilon}}\right )\phi+
	A\nabla \overline{u}\cdot \nabla \left ( \frac{u}{u_{\varepsilon}}\right )\phi\!\right)\! \dx= \\[2mm]
	\int_{\Gw}\!\left(\!\text{Re}\left (\frac{\overline{u}}{u_{\varepsilon}} f \right )\phi-\frac{2}{u_{\varepsilon}}(|\nabla u|_A^2-|\nabla u_{\varepsilon}|_A^2 )\phi\!\right)\!\! \dx\leq \int_{\Gw}\text{Re}\left (\frac{\overline{u}}{u_{\varepsilon}}f\right )\!\phi
	\dx.
	\end{multline*}
	Letting $\varepsilon\to 0$ we obtain \eqref{Kato_equa}.
	
	 The general case follows as in to the proof of Lemma~2.8 in \cite{Agmon}.
\end{proof}
An immediate corollary is the following lemma.
\begin{lemma}[{cf. \cite[Lemma~2.8]{Agmon}}]\label{cor_kato}
	Let $u\in  H^{1}_{\loc}(\wb)$ satisfy 
		$(P-\lambda,B)u=0$ in $\Gw$.
		Then $$(P,B)|u|\leq \mathrm{Re}(\lambda) |u|.$$
\end{lemma}
We use Lemma~\ref{u_minus} to obtain a seemingly stronger generalized maximum principle which holds for  $\lambda<\Gg(P,B,\Gw)$, where $\Gw$ is a Lipschitz bounded domain. 
Since $(P-\gl,B)$ is coercive for any $\gl<\Gl$, it follows that  $\Gl\leq \Gg(P,B,\Gw)$. 
\begin{lem}\label{utilz}
Let Assumptions~\ref{assump1} hold  in a domain $\Gw$, and assume that $\lambda<\Gg$. Then $(P-\lambda,B)$ satisfies the generalized maximum principle in $\Gw$.
\end{lem}
\begin{proof}
	We apply Agmon's method in \cite[Lemma~2.6]{Agmon}, where the Dirichlet problem is considered. 
		Fix $u\in H^1_{\loc}(\overline{\Gw}\setminus \pwd)$ with $u^{-}\in H^{1}_{\pwd}(\Gw)$ such that $(P-\lambda,B)u \geq 0$ in $\Gw$. We need to show that 
	$u^{-}=0$ in $\Gw$.
	
	Assume first that $\lambda < \Gl$.
	Lemma~\ref{u_minus} implies that
	$(P-\lambda,B)u^{-}\leq 0$ in $\Gw$, and therefore,
	$$
	\mathcal{B}_{P-\lambda,B}(u^{-},u^{-})\leq 0.
	$$
	On the other hand, the definition of $\Gl$ implies
	$$
	\mathcal{B}_{P-\lambda,B}(u^{-},u^{-})\geq 
	(\Gl-\lambda)\|u^- \|_{L^2}^{2},
	$$
	and therefore
	$u^{-}=0$.  So, the generalized maximum principle holds true for $(P-\lambda,B)$ in $\Gw$ for $\lambda < \Gl$.
	
	Next, specify $\lambda'<\Gl<\lambda<\Gg$, and let $u$ as above. 
	Then by Lemma~\ref{u_minus} we have,
	\begin{equation}\label{resol}
	(P-\lambda',B)u^{-}\leq (\lambda-\lambda')u^{-} \qquad \mbox{ in } \Gw.
	\end{equation}
	By  the generalized maximum principle for $(P-\lambda',B)$ in $\Gw$ (proved above), we have that 
		$$u^{-}\leq (\lambda-\lambda')R_{\Gw}(\lambda')u^{-} \qquad \mbox{in} \ \Gw.$$
		By \eqref{resol} it follows  that  $(P-\lambda')u^{-}\leq (\lambda-\lambda')^2R_{\Gw}(\lambda')u^{-}$, and by induction we obtain
	\begin{equation}\label{uminuslimit}
	u^{-}\leq (\lambda-\lambda')^kR_{\Gw}(\lambda')^{k}u^{-}\qquad  \forall k\in \N.
	\end{equation}
	The inequality $\| R_{\Gw}(\lambda')\|\leq
		\dfrac{1}{\mathrm{dist}(\lambda',\gs (\tilde{P},\tilde{B}))} \leq 
		\dfrac{1}{|\Gg-\lambda'|}$ implies
	$$
	|\lambda-\lambda'| \| R_{\Gw}(\lambda') \|\leq 
	\frac{|\lambda-\lambda'|}{|\Gg-\lambda'|}< q< 1.
	$$
	Letting $k \to 0$ in $\eqref{uminuslimit}$, we obtain $u^{-}=0$.
\end{proof}
\begin{rem}\label{imprem}
	\em{
	Suppose that $\Gg \notin \gs (\tilde{P},\tilde{B})$, then for $\gl'<\Gg$ we have $$ |\lambda'-\Gg| <  \mathrm{dist}(\lambda',\gs (\tilde{P},\tilde{B})) .$$ Therefore, the discreteness of $\gs (\tilde{P},\tilde{B})$ implies that   	
	there exists $\varepsilon>0$ such that for any $\lambda<\Gg+\varepsilon$ we have $|\lambda'-\lambda|\| R_{\Gw}(\lambda)\|<\tilde{q}<1$.  Repeating the proof of Lemma~\ref{utilz}, it  follows that $(P-\Gg-\varepsilon/2,B)$ satisfies the generalized maximum principle in $\Gw$. We note that in fact, $\Gg\in	\gs (\tilde{P},\tilde{B})$, see Theorem~\ref{agmonmethod} below.   
	}
\end{rem}
\subsection{Principal eigenvalue} In the present subsection we prove the existence of a principal eigenvalue in a
 Lipschitz bounded domain (cf. \cite{NP,BNV}).
\begin{defi}
	\em{
	We say that $\lambda_c$ is a {\em principal eigenvalue} of the operator $(\tilde{P},\tilde{B})$ in $\Gw$ and $u_c$ is its associated {\em principal eigenfunction}  if $u_c\in\mathcal{H}^{0}_{P-\lambda_c,B}(\Gw)\cap H^{1}_{\pwd}(\Gw)$.
}
\end{defi}
To establish the existence of $\lambda_c$ we use the following version of Krein-Rutman theorem.
\begin{theorem}[{Krein-Rutman-type theorem \cite[Corollary~2.2.3]{BI}}]\label{KTBIR}
	Let $$L^2_+(\Gw)=\{ f\in L^2(\Gw) \mid  f\geq 0\}.$$
	Suppose that $T:L^2(\Gw)\to L^2(\Gw)$ is a compact linear operator mapping $L^2_+(\Gw)$ into itself, and there exists $0\neq e\in L^2_+(\Gw)$ and a constant $\varrho>0$ such that $$
	(T-\varrho) e\in L^2_+(\Gw).
	$$  
	Then the spectral radius of $T$, denoted by $r(T)$, is positive  and there exists a nontrivial $v\in L^2_+(\Gw)$ satisfying $Tv=r(T)v$.
\end{theorem}
\begin{theorem}\label{geneignew}
	Let Assumptions~\ref{assump1} hold in a bounded domain $\Gw$. Then the operator $(\tilde{P},\tilde{B})$ admits a principal eigenvalue $\gl_c$ with a positive principal eigenfunction $u_c$. Hence,
$\mathcal{H}^{0}_{P-\lambda_c,B}(\Gw)\cap H^{1}_{\pwd}(\Gw)\neq \emptyset$.  In Particular, $\gl_c\leq \gl_0$.
\end{theorem}
\begin{proof}
	Consider the operator $T_\gl:=R_{\Gw}(\lambda)$, where $\lambda<\Gg$.
	By Corollary~\ref{cor_glc}, $T_\gl : L^2(\Gw)\to L^2(\Gw)$ is a compact operator. Moreover, the generalized maximum principle (which holds by Lemma~\ref{utilz}) implies that for all $f\in L^2_+(\Gw)$, we have $T_\gl f\in L^2_+(\Gw)$.
	Next, we let $0\lneqq \phi\in C_0^{\infty}(\Gw)$. 
	 By the generalized maximum principle and the strong maximum principle, $T_\gl\phi> 0$ in $\Gw$. Let 
	 $$
	 \varrho=\frac{1}{2\|\phi \|_{L^{\infty}(\Gw)}}\left (\inf\limits_{\text{supp}(\phi)}T_\gl \phi\right )>0.
	 $$
	By definition, $$T_\gl \phi-\varrho \phi=T_\gl\phi\geq 0 \qquad \text{~in~} \Gw \setminus \text{supp}(\phi).$$
	 Moreover,
	$$\varrho \phi=\frac{\phi}{2\|\phi \|_{L^{\infty}(\Gw)}}\left (\inf\limits_{\text{supp}(\phi)}T_\gl\phi \right )
	\leq T_\gl \phi \qquad\text{~in~} \text{supp}(\phi).
	$$
	 Hence, $T_\gl\phi-\varrho \phi\in L^2_+(\Gw)$. By Theorem~\ref{KTBIR}, $r(T_\gl)>0$ and there exists $0\lneqq u_c\in L^2(\Gw)\cap D(\tilde{P},\tilde{B})$ satisfying
	 $
	 T_\gl u_c=r(T_\gl)u_c.
	 $
	 in $\Gw$. 
	 Therefore, $u_c\in H^{1}_{\pwd}(\Gw)$ and  
	 $$Pu_c=\left (\dfrac{1}{r(T_\gl)}+\lambda\right )u_c, \quad  Bu=0 \mbox{ on } \pwr.$$
	  By the strong maximum principle, $u_c>0$ in $\Gw$. Thus, $\gl_c=\gl_c(\gl)= r(T_\gl)^{-1}+\lambda$ is a principal eigenvalue of $(\tilde{P},\tilde{B})$.
	  \end{proof}
  \begin{remark}{\em
  	The  simplicity of the principal eigenvalue of $(P,B)$ in $\Gw$, when $(P,B)$ is nonselfadjoint and nonsmooth, remains open even in the case $\pwr=\emptyset$.
}
  \end{remark}
We have
\begin{theorem}\label{agmonmethod}
	Let Assumptions~\ref{assump1} hold in a bounded domain $\Gw$. Then $\Gg=\lambda_c\in \gs (\tilde{P},\tilde{B})$. In particular, $\gl_c$ does not depend on $\gl$.
	
	Moreover, for any $\lambda\in \R$, the operator $(P-\lambda,B)$  satisfies the generalized maximum principle in $\Gw$ if and only if $\lambda<\Gg=\lambda_c$.
\end{theorem}
For the proof of Theorem~\ref{agmonmethod} we need the following auxiliary lemma.
  \begin{lemma}\label{2direction}
  	Let Assumptions~\ref{assump1} hold in a bounded domain $\Gw$, and
  	assume that the generalized maximum principle for $(P-\lambda,B)$ holds in $\Gw$ for some $\lambda\in \R$.
  	Then $\lambda \leq \Gg$.
  \end{lemma}
  \begin{proof}[Proof of Lemma~\ref{2direction}]
It is enough to show that if $(P-\lambda,B)$ satisfies the generalized maximum principle in $\Gw$, then for any eigenvalue $z\in \sigma(\tilde{P},\tilde{B})$ we have $\gl <\text{Re}(z)$. Indeed,  let $u\in H^{1}_{\pwd}(\Gw)$, $u\neq 0$, satisfy $(P-z,B)u=0$ with $\text{Re}(z)\leq \lambda$. By Lemma~\ref{cor_kato}, $(P-\lambda,B)|u| \leq 0$, and therefore, $(P-\lambda,B)(-|u|)\geq 0$. By the generalized maximum principle, $u\equiv 0$, which is a contradiction.
  \end{proof}
  \begin{proof}[Proof of Theorem~\ref{agmonmethod}]
  	First, we claim that $\Gg\in \sigma (\tilde{P},\tilde{B})$.
  	Otherwise, Remark~\ref{imprem} implies now that there exists $\varepsilon>0$ such that
  	$(P-\Gg-\varepsilon,B)$  satisfies the generalized maximum principle in $\Gw$, a contradiction to Lemma~\ref{2direction}. 
  	
  	Corollary~\ref{cor_glc} implies that $\Gg=\mu^{-1}+\lambda$, where $\lambda$ is sufficiently small such that $(P-\lambda,B)$ is coercive in $\Gw$ for any $\gl'\leq \gl$, and $\mu\neq 0$ is a real eigenvalue of $R_{\Gw}(\lambda)$. Since $\gl$ belongs to the resolvent set, $\mu>0$, and we have  
  	$$
  	\Gg=\mu^{-1}+\lambda \geq \frac{1}{r(R_{\Gw}(\lambda))}+\lambda=\lambda_c\, .
  	$$  
  	By the definition of $\Gg$ we also have $\Gg\leq \lambda_c$. Hence, $\Gg=\lambda_c$,   and  $\gl_c$ does not depend on $\gl$.
  	
  	 Since $\Gg=\lambda_c$, it follows that the generalized maximum principle does not hold for $\gl= \Gg$. In addition,  Lemma~\ref{2direction} implies that the generalized maximum principle does not hold for  $\gl > \Gg$. On the other hand, by Lemma~\ref{utilz}, the generalized maximum principle holds for $\gl<\Gg$, and the second assertion of the theorem follows. 
  \end{proof}
 Corollary~\ref{cor_gwmp} and Lemma~\ref{2direction} imply the following.
\begin{lem}\label{2.12}
	 Let Assumptions~\ref{assump2} hold in a domain $\Gw$ and  let  $u\in \mathcal{H}^{0}_{P,B}(\Gw)$, or more generally, assume that $u$  is a regular positive supersolution of \eqref{P,B} in $\Gw$. 
	 Then:
	 
	 \medskip
	 
	 (1) for any  Lipschitz bounded domain $\Gw_0\!\Subset_R \!\Gw$, we have  $0\!<\! \lambda_c(P,B,\Gw_0) \!\leq\! \lambda_0(P,B,\Gw_0)$;

	(2) if $\{\Gw_k\}_{k\in \N}$ is an exhaustion of $\wb$, then  the sequences $\{\!\lambda_c(P,B,\Gw_k)\!\}_{k\in \N}$,  and $\{\!\lambda_0(P,B,\Gw_k)\!\}_{k\in \N}$ are strictly decreasing. 
\end{lem}
\begin{proof} 
(1) We claim that $\inf \limits_{\overline{\Gw_0}}u>0$. 
 Indeed, let $\Gw_0\Subset_R \tilde \Gw \Subset_R \Gw$ be a Lipschitz subdomain.
For any $x\in\tilde \Gw\cap \overline{\Gw_0}$ there exists an open neighborhood $x\in B_x\in \tilde \Gw$ such that $\inf\limits_{B_x}u\geq C_x>0.$ Since $\Gw_0$ is precompact we may subtract a finite subcover $\{B_{x_j}\}_{j=1}^{m}$ from which we obtain
$\inf \limits_{\overline{\Gw_0}}u>\min\limits_{1\leq j \leq m}C_{x_{j}}>0$.	   By Lemma~\ref{gen_max_weak}, $(P,B)$ satisfies the generalized maximum principle in $\Gw_0$, which in light of Theorem~\ref{agmonmethod} implies that   $\lambda_c(P,B,\Gw_0) > 0$. 
	Part (2) follows from Part (1) and Lemma~\ref{gen_max_weak}.
\end{proof}
We conclude this section with the following lemma.
\begin{lem}\label{lem-lim}
		Let Assumptions~\ref{assump2} hold in  $\Gw$, and let $\{\Omega_k\}_{k\in \mathbb{N}} \subset \Omega$ 
	be an exhaustion  of $\wb$. Denote $\tilde\gl_0(\Gw):=\lim_{k\to \infty} \gl_0(\Gw_k)$, and $\tilde\gl_c(\Gw):=\lim_{k\to \infty} \gl_c(\Gw_k)$. Then
	$$\gl_0(\Gw)=\tilde\gl_0(\Gw)=\tilde\gl_c(\Gw).$$
	
	Assume further that Assumptions~\ref{assump1} hold in  $\Gw$, and that 
	$\lim_{k\to \infty} u_{c,k} =u_{c}$, where $u_{c,k}$ and $u_{c}$  are  principal eigenfunctions in $\Gw_k$ and $\Gw$ respectively. 
Then 
$$\gl_0(\Gw)=\gl_c(\Gw)= \Gg(P,B,\Gw).$$
\end{lem}
\begin{proof}
	By Lemma~\ref{2.12}, the sequences $\{ \gl_c(\Gw_k)\}_{k\in \N}$ and $\{ \gl_0(\Gw_k)\}_{k\in \N}$ are monotone decreasing and satisfy 
	$$ \gl_0(\Gw) <  \gl_c(\Gw_k)\leq  \gl_0(\Gw_k).$$ 
	Let $u_{c,k}\in \mathcal{H}^{0}_{P-\lambda_c(\Gw_k)}$ be a principal eigenfunction satisfying $u_{c,k}(x_0) = 1$, and let $u_{k}\in \mathcal{H}^{0}_{P-\lambda_0(\Gw_k)}$ satisfy $u_{0}(x_0)=1$.  Then the Harnack convergence principle implies that 
	$\gl_0(\Gw)\leq \tilde\gl_c(\Gw)\leq \tilde\gl_0(\Gw)$. On the other hand, by the definition of $\gl_0(\Gw)$ we have $\tilde\gl_c(\Gw)\leq \tilde\gl_0(\Gw)\leq \gl_0(\Gw)$.
	
	Under the further assumptions we clearly have $\tilde\gl_c(\Gw)=\gl_c(\Gw)$, and by the first part, we have, $\gl_0(\Gw)=\gl_c(\Gw)= \Gg(P,B,\Gw)$.
\end{proof}
\section{Criticality theory}\label{sec_crit}
In the present section we discuss a criticality theory for $(P,B)$ in a domain $\Gw$. More precisely, we study the relation between the validity of the generalized maximum principle of $(P,B)$ in bounded Lipschitz subdomains of $\Gw$, the existence of positive (super)solutions of $(P,B)u=0$ in $\Gw$, and the nonnegativity of the generalized principal eigenvalue. Moreover, we define the notion of positive solutions of minimal growth at infinity in $\Gw$, criticality and subcriticality of $(P,B)$ in $\Gw$ and discuss some related properties.
 Unless otherwise stated,
{\bf we assume throughout the section that Assumptions~\ref{assump2}  are satisfied in $\Gw$}. 
\subsection{Characterization of $\gl_0$}
Let $\gl_0$ be  the generalized  principal eigenvalue (see Definition~\ref{def_princ}). As a consequence of the results in the previous section, we obtain the following characterization of $\gl_0$.
\begin{theorem}\label{eeqiuv_add}
	The following assertions are equivalent:
	\begin{enumerate}
		\item  $ {\mathcal H}^0_{P,B}(\Omega)\neq \emptyset$. In other words, $\lambda_0(P,B,1,\Gw)\geq 0$.
		\item  \eqref{P,B} admits a regular positive supersolution in $\Gw$.
		\item $\lambda_0(P,B,1,\Gw')>0$ for any Lipschitz subdomain $\Gw'\Subset_R\Gw$.
		\item $(P,B)$ satisfies the generalized maximum principle in
		any Lipschitz subdomain  $\Gw'\Subset_R \Gw$.
	\end{enumerate}
\end{theorem}
\begin{proof}
	$(1) \Longrightarrow (2):$ Obvious. \\
	$(2) \Longrightarrow (3):$ Follows from Lemma~\ref{2.12} \\
	$(3) \Longrightarrow (4):$ Follows from Corollary~\ref{cor_gwmp} and Lemma~\ref{324}. \\
	$(4) \Longrightarrow (1):$ 
		By Lemma~\ref{2direction}, $\Gg(P,B,1,\Gw')\geq 0$ for any Lipschitz subdomain $\Gw'\Subset_R \Gw$. Since the generalized maximum holds for any $\gl<\Gg$, it follows that for any $\gl<0$ the resolvent operator is positive. 
	Let $\{ \Gw_k\}_{k\in \N}$ be an exhaustion of $\wb$ and fix $x_0\in \Gw_1$. 
	
	For $k\geq 1$, denote by $R_k$ the resolvent operator of $(\tilde P+1/k,\tilde B)$ in $\Gw_k$, and let $f_k\in C_0^\infty(\Gw_k\setminus \Gw_{k-1})$ be a nonzero nonnegative function.
	Using the Harnack convergence principle (Lemma~\ref{HCP}), it follows that the sequence 
	$$\left\{\frac{R_k(f_k)}{R_k(f_k)(x_0)}\right\}_{k\in \N}$$ admits a subsequence converging to $u\in \mathcal{H}^{0}_{P,B}(\Gw)$.
\end{proof}
Theorem~\ref{eeqiuv_add} clearly implies the strict monotonicity of the generalized principal eigenvalues in bounded subdomains (cf. Lemma~\ref{lem-lim}). 
\begin{cor}\label{stmon}
	Let $\Gw_1$ and $\Gw_2$ be two Lipschitz nonempty subdomains of $\Gw$ satisfying
$\Omega_1\Subset_R \Omega_{2}\Subset_R \Gw$.	Then $$\lambda_0(P,B,\Gw)< \lambda_0(P,B,\Gw_2)<\lambda_0(P,B,\Gw_1). $$
In particular, if $\{ \Gw_k\}_{k\in \N}$ is an exhaustion of   of $\wb$, then $\lambda_0(\Gw_k)\to \lambda_0(\Gw)$.
\end{cor}
\begin{proof}
	 By the monotonicity of $\gl_0$ with respect to increasing subdomains domain, it follows that $\lambda_0(\Gw)\leq \gl_1$, where  $\gl_1:=\lim_{k\to \infty}\lambda_0(\Gw_k)$. We  need to prove that $\lambda_0(\Gw)=\gl_1$. Suppose that $\lambda_0(\Gw)<\gl_1$, then by the Harnack convergence principle  (Lemma~\ref{HCP}), it follows that $\mathcal{H}^{0}_{P-\gl_1,B}(\Gw)\neq \emptyset$, which is a contradiction to the definition of   $\lambda_0(\Gw)$. 
	\end{proof}
As a consequence of Theorem~\ref{eeqiuv_add} we obtain the existence of positive (super)solutions  to the following mixed boundary value problem. 
\begin{lem}\label{lem1.8}
	Assume that $(P,B)\geq 0$ in $\Gw$, and let $\Gw'\Subset_R \wb$ be a  Lipschitz subdomain of $\Gw$. 
	  Let $K\Subset \Gw'$ be a Lipschitz  subdomain. 
	  
	   Then for any nonzero nonnegative function $f\in C_0^\infty(\Gw'\setminus K)$ there exists a unique positive weak solution $u$  to  the problem 
	\begin{equation}\label{eq: with K}
	\begin{cases}
	Pw=f & \Omega'\setminus K , \\
	Bw= 0  & \partial \Gw'_{\mathrm{Rob}},\\
	\mathrm{Trace}(w)=0&  \partial\Gw'_{\mathrm{Dir}} \cup \partial K.
	\end{cases}
	\end{equation}
\end{lem}
\begin{proof}
	By Theorem~\ref{agmonmethod}, $$\gl_c(P,B,1,\Gw'\setminus K)=\Gg(P,B,\Gw'\setminus K)>0.$$ Therefore, $(\tilde P,\tilde B)$ is invertible in $\Gw'\setminus K$,  and the corresponding resolvent operator is positive. Hence, there exists a  unique nonnegative solution 
	$u\in H^{1}_{(\partial {\Gw'}\cup \partial K)\setminus \pwr}(\Gw'\setminus K)$ to problem
\eqref{eq: with K}. 
	The strict positivity of $u$ follows from the strong maximum principle. 
\end{proof}
\subsection{Minimal Growth}  
Next, we introduce the notion of positive solution of minimal growth for \eqref{P,B} (cf. \cite{Agmon,PS}). In the sequel $\{ \Gw_k\}_{k\in \N}$ is an exhaustion of $\wb$.

\begin{defi}\label{def:minimalgrowth}
	\em{
		A function $u$ is said to be a {\em positive solution of $(P,B)$ of minimal growth in a neighborhood of infinity in $\Gw$}  if 
		$u\in \mathcal{H}^{0}_{P,B}(\Gw^*_j)$ for some $j \geq 1$ and for any $l>j$ and $v\in C(\Gw^*_l \cup \Pw_{l,\Dir})\cap \mathcal{SH}_{P,B}(\Gw^*_l)$,  $u \leq v$ on 
		$\Pw_{l,\Dir}\;\; \Longrightarrow  \;\; u \leq v$ on $\Gw^*_l$.
	}
\end{defi}

\begin{lem}\label{lem:min_growth}
Assume that $(P,B)\geq 0$ in $\Gw$.  Then for any $x_0\in \Gw$ the equation $(P,B)u=0$ admits (up to a multiplicative constant) a unique positive solution $v$ in $\Gw \setminus \{ x_0\}$ of minimal growth in a neighborhood of infinity in $\Gw$.
\end{lem}
\begin{proof}
	{\bf Existence:} Fix an exhaustion $\{\Gw_k\}_{k\in \N}$ of $\wb$.  We may assume that $x_0\in \Gw_1$.  For $k\in \N$, let $B_k=B(x_0,\gd/k)$, where $\gd>0$ is sufficiently small such that $B_1\Subset \Gw_1$. Let $f_k\in C_0^\infty(B_{k-1}\setminus B_{k})$ be a nonzero nonnegative function. 
	By Lemma~\ref{lem1.8}, there exists a unique   positive solution $v_k$ to the problem	
	$$
	\begin{cases}
	Pw=f_k & \Omega_k\setminus B_k , \\
	Bw=0 &  \Pw_{k,\mathrm{Rob}},\\
	\mathrm{Trace}(w)=0& \partial \Gw_{k,\mathrm{Dir}}\cup \partial B_k.\\
	\end{cases}
	$$
	Fix $x_1\in \Gw\setminus B_1$, and consider the sequence $\{u_k:= v_k/v_k(x_1) \}_{k\in \N}$. 
	By the Harnack convergence principle  (Lemma~\ref{HCP}),  $\{u_k \}_{k\in \N}$  converges (up to a subsequence) to a positive solution $u\in H^{1}_{\loc}((\wb)\setminus \{x_0\})$ of $(P,B)u=0$ in $\Gw\setminus \{x_0 \}$.

\medskip 
	  
	We claim that $u$ is a positive solution of minimal growth in a neighborhood of infinity in $\Gw$. 
	Indeed, fix $l>1$ and let  
	$s\in  \mathcal{SH}_{P,B}(\Gw^*_l)\cap C(\Gw^*_l \cup \Pw_{l,\Dir})$ such that  $u\leq s$ on $\Pw_{l,\Dir}$. For each $k>l$ the boundary condition  $\mathrm{Trace}(u_{k})=0$  on $\Pw_{k,\Dir}$, and the generalized maximum principle  (Lemma~\ref{utilz}) imply that $ u_{k}\leq s$ in $\Gw_{k}\setminus \Gw_l$.  By letting $k\to \infty$, we obtain $u\leq s$ in $\Gw_l^{*}$. 

\medskip 

	{\bf Uniqueness:} Let $u$ and $v$ 
	be positive solutions of $(P,B)$ in $\Gw\setminus \{ x_0\}$ having minimal growth at infinity in $\Gw$.
Clearly, $u$ has a removable singularity at $x_0$ if and only if $v$ has.  
If the singularity at $x_0$ is removable, we may assume that $u(x_0)=v(x_0)$, and a simple comparison argument implies that $u=v$. Otherwise, by \cite[Theorems 1 and 5]{Serrin}, $v\sim u\sim G_P^{B_1}(x,x_0)$ near $x_0$, where $G_P^{B_1}(x,x_0)$ is the positive (Dirichlet) minimal Green function of $P$ in $B_1$, and again, a comparison argument implies that $u=Cv$ with $C>0$.  
\end{proof}
\subsection{Criticality vs. subcriticality} For a nonnegative operator $(P,B)$ we introduce the notions of criticality and subcriticality.
\begin{defi}\label{def:4.6}
	\em{
Assume that $(P,B)\!\geq  \! 0$ in $\Gw$.  The operator $(P,B)$ is {\em subcritical}  in $\Gw$ if there exists $0\lneqq W\in L^{p}_\loc(\wb)$, $p>n/2$, such  that  $\mathcal{H}^{0}_{P-W,B}(\Gw)\neq  \emptyset$; such a $W$ is called a {\em Hardy-weight} for $(P,B)$ in $\Gw$, Otherwise, $(P,B)$ is said to be {\em critical} in $\Gw$. 

If $\mathcal{H}^{0}_{P,B}(\Gw)= \emptyset$, then  $(P,B)$  is called {\em supercritical} in $\Gw$.
	}
\end{defi}
\begin{lem}\label{inview}
	Let $(P,B)  \geq   0$ in $\Gw$.  Then $(P,B)$ is critical in $\Gw$ if and only if $(P,B)$ admits (up to a multiplicative constant) a unique  regular positive supersolution in $\Gw$. 
\end{lem}
\begin{proof}
	Assume that $(P,B)$ is subcritical in $\Gw$ and let $W\gneqq 0$, $W\in L^{p/2}_\loc(\wb)$ be a Hardy-weight.  Then any  $u\in \mathcal{H}^{0}_{P,B}(\Gw)$ and $v\in \mathcal{H}^{0}_{P-W,B}(\Gw)$ are regular positive supersolutions of $(P,B)$ in $\Gw$ which are linearly independent.
	
	If $(P,B)$ has two linearly independent regular positive supersolutions, $v$ and $u$ in $\Gw$, then a direct calculation (see \cite[Lemma~5.1]{DFP}) shows that 
	$\sqrt{uv}\in \mathcal{RSH}_{P-W,B}(\Gw)$, where 
	$W=\frac{1}{4}\left |\nabla (\log(v/u))\right |_A^2\gneqq 0$ is the corresponding Hardy-weight. 
	Clearly,  $W\in L^1_{\loc}(\Gw)$, and therefore there exists  $0 \lneqq W'\leq W$ such that $W'\in C^\infty_0(\Gw)$. For such $W'$ we have that  $\sqrt{uv}\in \mathcal{RSH}_{P-W',B}(\Gw)$, and in light of Theorem~\ref{eeqiuv_add}, $(P,B)$ is subcritical in $\Gw$.
\end{proof}
\begin{rem}
{\em 	Lemma~\ref{inview} and the proof  of Lemma~\ref{gen_max_weak} clearly  imply that 
	$(P,B)$ is critical (resp., subcritical) in $\Gw$ if and only if $(P^u,B^u)$ is critical (resp., subcritical) in $\Gw$, where $u$ is a regular positive supersolution of $(P,B)$ in $\Gw$.
}	
\end{rem}
\begin{defi}
	\em{
		A function $u\in \mathcal{H}^{0}_{P,B}(\Gw)$ is called a {\em ground state} if $u$ has minimal growth in a neighborhood of infinity in $\Gw$.
	}
\end{defi}
\begin{lemma}\label{gs_implies_crit}
	Assume that $(P,B)\geq 0$ in $\Gw$. Then
	 $(P,B)$ admits a ground state in $\Gw$ if and only if $(P,B)$ is critical in $\Gw$.  
\end{lemma}
\begin{proof}
	Suppose that   $(P,B)$ admits a ground state $\vgf$ in $\Gw$. Let $\{ \Gw_k\}_{k\in \N}$ be an exhaustion of $\wb$, and let $x_1\in \Gw_1$.
	Let $v$ be  a regular positive supersolution of $(P,B)$ in $\Gw$.  In light of Lemma~\ref{inview}, it suffices to show that $v=C\vgf $ for some constant $C>0$.\\
		Since $\vgf$ has minimal growth, it follows that there exists $C>0$ satisfying 
	$C\vgf\leq v$ in $\Gw$.  Let 
	$$C_0:=\sup\{C>0 \mid C\vgf \leq v \quad  \text{ in } \Gw \}.$$
	Consider the function $v_0:=v-C_0\vgf \geq 0$ and note that  $(P,B)v_0 \geq 0$ in $\Gw$. Assume by contradiction that  $v_0\gneqq 0$. By the strong maximum principle (Lemma~\ref{maxagmon}, $v_0>0$ in $\Gw$, and therefore, $v_0\in \mathcal{RSH}_{P,B}(\Gw)$. By repeating the above argument with $\vgf$ and $v_0$, it follows that there  exists $\mu>0$ such that 
	$\mu \vgf\leq v_0$. Hence,  
	$v\geq (C_0+\mu)\vgf$, a contradiction to the maximality of $C_0$.
	Hence, $v=C_0\vgf$.
	
	Assume now that $(P,B)$ does not admit a ground state in $\Gw$.
	 Let $u_{x_0}$ be the positive solution of minimal growth having a nonremovable singularity at $x_0$, and $\{\Gw_k\}_{k\in \N}$ be an exhaustion of $\wb$ (see Lemma~\ref{lem:min_growth}).
	  Let $0\lneqq f\in C_0^{\infty}(\tilde \Gw)$, where $\tilde\Gw$ is a small neighborhood of $x_0$,  and  
		let $w_k$  be the solution of the problem:
		$$
		\begin{cases}
		Pw=f & \Omega_k, \\
		Bw=0 & \Pw_{k,\mathrm{Rob}},\\
		\mathrm{Trace}(w)=0& \partial {\Gw_{k,\Dir}}.
		\end{cases}
		$$
		By the generalized maximum principle $\{w_k\}_{k\in \N}$ is monotone increasing. If there exists $x_1\neq x_0$ such that the sequence
		$\{w_k(x_1)\}_{k\in \N}$ is bounded, then 
		by Harnack convergence principle  (Lemma~\ref{HCP}) and elliptic regularity,  $\{w_k\}_{k\in \N}$ converges locally uniformly in $\Gw $ to a positive solution $w$ of $(P,B)u=f\gneqq 0$ in $\Gw$. Clearly $w$ is a regular positive supersolution of \eqref{P,B} in $\Gw$. In light of Lemma~\ref{inview}, $(P,B)$is subcritical in $\Gw$. 
		
		Otherwise, fix $x_1\neq x_0$. Then once more, by the Harnack convergence principle  and interior elliptic regularity, the sequence $\{z_k:= w_k/w_k(x_1) \}_{k\in \N}$ converges to a positive solution $z\in \mathcal{H}^{0}_{P,B}(\Gw)$.
		Note that $z<u_{x_0}$ near $x_0$ (since $u_{x_0}$ has a singularity at $x_0$), and therefore, by a standard comparison argument,
		 $z$ is smaller than $u_{x_0}$ in $\Gw\setminus \{x_0\}$. Hence, $z$ is a positive solution of of $(P,B)u=0$ in $\Gw$ of minimal growth at infinity.  Namely, $z$ is a ground state of $(P,B)$ in $\Gw$, a contradiction to our assumption.   
\end{proof}
\section{Positive minimal Green function}\label{sec_green}
\subsection{Green function in a Lipschitz bounded subdomain}
Throughout the present subsection, 
{\bf we assume that $(P,B)$ is nonnegative in a domain $\Gw$ and that Assumptions~\ref{assump2} hold true.}  
We construct in the subcritical case the corresponding positive minimal Green function $G^{\Gw}_{P-\lambda,B}(x,y)$ of $(P,B)$ in $\Gw$.  We follow Stampacchia's Green function construction for the Dirichlet boundary value problem in a bounded domain \cite[Section 9]{ST}.
We mention the following result for subsequent applications (cf. \cite[Theorem~3.9]{Adams}).
\begin{proposition}
	Let $\Gw'\Subset_R \Gw$ and $q> 1$. Then
	$\Upsilon\in \big(W^{1,q}_{\partial \Gw'_{\mathrm{Dir}}}(\Gw')\big)^*$
	if and only if 
	there exist 
	$g_0\in L^{q'}(\Gw')$ and $\mathbf{g}\in L^{q'}(\Gw',\R^n)$ such that for any $u\in W^{1,q}_{\partial \Gw'_{\mathrm{Dir}}}(\Gw')$,
	\begin{equation}\label{eq:dual_sobolev}
	\Upsilon(u)=\int_{\Gw'}(ug_0- \nabla u\cdot \mathbf{g}) \dx,
	\end{equation}
and in this case we  write 	$\Upsilon=g_0+\diver\mathbf{g}$. 

Moreover, there exist $g_0\in L^{q'}(\Gw')$ and $\mathbf{g}\in L^{q'}(\Gw',\R^n)$ satisfying  \eqref{eq:dual_sobolev} such that 
	$$\|\Upsilon\|_{\big(W^{1,q}_{\partial \Gw'_{\mathrm{Dir}}}(\Gw')\big)^{\!*}}  \asymp  \left( \|g_0\|_{L^{q'}(\Gw')}\!+\!\|\mathbf{g}\|_{L^{q'}(\Gw',\R^n)} \right).$$
\end{proposition}
Let $\Gw'\Subset_R \Gw$ be a Lipschitz subdomain of $\Gw$. Then  for $\gl< \lambda_c(\Gw')$ the resolvent operator $R_{\Gw'}(\lambda)$ exists, and  $R_{\Gw'}(\lambda):L^{2}(\Gw')\to H^1_{\partial\Gw'_{\mathrm{Dir}}}(\Gw')$ is  bounded. 
By the Fredholm alternative, for any $\Upsilon=g_0+\diver\mathbf{g} \in (H^1_{\partial \Gw'_{\mathrm{Dir}}}(\Gw'))^*$ there exists a unique solution $u_{\Upsilon}\in H^1_{\partial \Gw'_\mathrm{Dir}}(\Gw')$ to the problem
\begin{equation}\label{eq:Rob_func}
\begin{cases}
(P-\gl)w=g_0+\diver\mathbf{g} & \quad \text{in } \Gw', \\
Bw= - \mathbf{g}\cdot \vec{n} & \quad \text{on } \partial \Gw'_{\mathrm{Rob}},
\end{cases}
\end{equation}
 (in short, $(P-\gl,B)w \!= \! g_0+\diver \mathbf{g}$ in $\Gw'$) in the following sense:  
 For all $\phi \! \in \! \mathcal{D}(\Gw',\partial \Gw'_{\mathrm{Dir}})$,
$$
\int_{\Gw'}\!\!\left[\! A\nabla u_{\Upsilon}\!\cdot\! \nabla \phi+u_{\Upsilon}\bt\!\cdot\! \nabla \phi
+\bb \!\cdot\!\nabla u_{\Upsilon} \phi\!+\!(c-\gl)u_{\Upsilon}\phi \right]\!\!\!\dx+\!\int_{\partial \Gw'_{\mathrm{Rob}}}\!\!\frac{\gamma}{\beta}u_{\Upsilon}\phi \dsigma\!=\!\!\int_{\Gw'}\!( g_0 \phi - \mathbf{g}\cdot \nabla\phi) \!\dx.
$$
 Note that in Theorem~\ref{coercivity} we proved the unique solvability of \eqref{eq:Rob_func} only for the case  $\mathbf{g}=0$, but the proof applies also  for the general case $\mathbf{g} \neq 0$ (see for example, \cite[Theorem 8.3]{GT}).
\begin{lem}\label{closed_operator}
	 Assume that 
		$(P,B)\geq 0$ in $\Gw$, and let $\Gw'\Subset_R \Gw$ be a Lipschitz subdomain. 
	Let $F:(H^1_{\partial \Gw'_\mathrm{Dir}}(\Gw'))^*\to
	H^1_{\partial \Gw'_\mathrm{Dir}}(\Gw')
	$
	be the linear operator mapping $\Upsilon \mapsto u_{\Upsilon}$, where $(P,B)u_{\Upsilon}=\Upsilon$ in $\Gw'$.  Then $F$ is continuous.
\end{lem}
\begin{proof}
	It is enough to show that the graph of $F$ is closed.
	Let $(\Upsilon_k,u_k)\to (\Upsilon,u)$ in the graph norm, where $F(\Upsilon_k)=u_k$. We need to show that $F(\Upsilon)=u$, that is  $(P,B)u=\Upsilon$.
	By the definition of $F$, for any $\phi\in \mathcal{D}(\Gw',\partial \Gw'_{\mathrm{Dir}})$
we have 	$\mathcal{B}_{P,B}(u_k,\phi)=
	\Upsilon_k(\phi)$.
	Clearly, $\Upsilon_k(\phi)\to \Upsilon(\phi)$, and by Theorem~\ref{coercivity} 
	$$\mathcal{B}_{P,B}(u_k,\phi) \to \mathcal{B}_{P,B}(u,\phi) \qquad  \mbox{as } k\to \infty. $$
	Hence, $F(\Upsilon)=u$. 
\end{proof}
The following lemma utilizes standard tools of functional analysis to obtain the Green function of $(P,B)$ in a Lipschitz subdomain $\Gw'\Subset_R \Gw$. Recall that according to 
Assumptions~\ref{assump2} the exponent $p$ satisfies $p>n$.
\begin{lemma}\label{bound grad2}
 Suppose that $(P,B)\geq 0$ in $\Gw$, and let  $\Gw'\Subset_R \wb$  be a  Lipschitz subdomain. 
	Let $\lambda<\lambda_c(\Gw')$, and let $u_{\Upsilon}\in H^1_{\partial \Gw'_\mathrm{Dir}}(\Gw')$ be the unique weak solution of \eqref{eq:Rob_func},  where $g_0\in L^p(\Gw'), \mathbf{g}\in L^p(\Gw',\R^n)$.
	Then for any $\Gw'' \Subset_R \Gw'$ 
	\begin{equation}\label{rety}
	\sup\limits_{\Gw''}|u_\Upsilon|\leq C(\|g_0 \|_{L^p(\Gw')}+\|\mathbf{g} \|_{L^p(\Gw',\R^n)}),
	\end{equation}
	where $C$ depends only on $\Gw',\Gw'',\gl$ and the  coefficients of the operator $(P,B)$ in $\Gw$. 

In particular, for any $x\in \Gw''$, the functional 
$$J_x:(W^{1,p'}_{\partial \Gw'_{\mathrm{Dir}}}(\Gw'))^* \to \R,\qquad \Upsilon \mapsto u_{\Upsilon}(x)$$ 
is bounded, and there exists  $0<G_{P-\lambda,B}^{\Gw'}(x,\cdot)\!\in \!(W^{1,p'}_{\partial \Gw'_{\mathrm{Dir}}}(\Gw'))^{**} \! \subset \! W^{1,p'}_{\partial \Gw'_{\mathrm{Dir}}}(\Gw')$ satisfying
$J_x(\Upsilon)\!=\!G_{P-\lambda,B}^{\Gw'}(x,\cdot)(\Upsilon)$ for all $\Upsilon\! \in \!
(W^{1,p'}_{\partial \Gw'_{\mathrm{Dir}}}(\Gw'))^{*} $,  and the following Green representation formula holds: 
\begin{equation}\label{lateG1}
u_\Upsilon(x)= \int_{\Gw'}\big[G^{\Gw'}_{P-\lambda,B}(x,y)g_0(y)-\nabla_y G^{\Gw'}_{P-\lambda,B}(x,y)\cdot \mathbf{g}(y)\big]\! \dy. 
\end{equation} 
\end{lemma}
\begin{proof}
 Denote $u=u_ \Upsilon$. It follows from  \cite[Theorem~5.36]{L}  that
$$\sup\limits_{\Gw''} |u|\leq C( \| g_0\|_{L^p(\Gw')}+\|\mathbf{g} \|_{L^p(\Gw',\R^n)})+\|u\|_{L^2(\Gw')}).$$ 
On the other hand, Lemma~\ref{closed_operator} implies that  
$$\|u\|_{L^2(\Gw')}\leq
C\| \Upsilon\|_{\big(W^{1,2}_{\partial \Gw'_{\mathrm{Dir}}}(\Gw')\big)^*}\
\leq 
C\| \Upsilon\|_{\big(W^{1,p'}_{\partial \Gw'_{\mathrm{Dir}}}(\Gw')\big)^*}\asymp \|g_0 \|_{L^p(\Gw')}+\|\mathbf{g} \|_{L^p(\Gw',\R^n)}.$$

Moreover, in light of the generalized and the strong maximum principle,  if $(P,B)u\gneqq 0$ in $\Gw'$, then $u>0$ in $\Gw'$.
Since the space $W^{1,p'}_{\partial \Gw'_{\mathrm{Dir}}}(\Gw')$ is a closed subspace of a  reflexive Sobolev space, it is reflexive as well.
In particular, there exists $0< G^{\Gw''}_{P-\lambda,B}(x,\cdot)\in W^{1,p'}_{\partial \Gw'_{\mathrm{Dir}}}(\Gw')$  such that for any $\Upsilon=g_0+\diver\mathbf{g}$,
\begin{equation}\label{lateG}
J_x(\Upsilon)=\Upsilon(G^{\Gw''}_{P-\lambda,B}(x,\cdot))=\int_{\Gw'}\big(G^{\Gw''}_{P-\lambda,B}(x,y)g_0(y)-\nabla_y G^{\Gw''}_{P-\lambda,B}(x,y)\cdot \mathbf{g}(y)\big)\! \dy.
\end{equation}
The reflexivity also implies that $$G_{P-\lambda,B}^{\Gw''}(x,\cdot)(\Upsilon)=\Upsilon(G_{P-\lambda,B}^{\Gw''}(x,\cdot))=J_x(\Upsilon).$$
 Moreover, it follows that  for each fixed $x\in \Gw'$, $G_{P-\lambda,B}^{\Gw''}(x,\cdot)$ does not depend on subdomains $\Gw''$ containing $x$. Therefore, we obtain \eqref{lateG1}.
	\end{proof}
\begin{cor}
	Fix $x_0\in \Gw''\Subset_R \Gw'\Subset_R \Gw$, where $\Gw'',\Gw'$ are Lipschitz subdomains of $\Gw$.
	Then the operator $H:(W^{1,p'}_{\partial \Gw'_{\mathrm{Dir}}}(\Gw'))^{*}\to L^{\infty}(\Gw'')$, which
	maps $\Upsilon$ to $u_{\Upsilon}$, is bounded. Moreover,
	 $H^*:L^1(\Gw'')\to (W^{1,p'}_{\partial \Gw'_{\mathrm{Dir}}}(\Gw'))^{**}$ is bounded as well.
\end{cor}
Next, we are interested in the regularity properties of the function $G_{P-\lambda ,B}^{\Gw'}(x,y)$. It turns out  that $G_{P-\lambda ,B}^{\Gw'}(x,y)$ satisfies $(P-\gl,B)G_{P-\lambda ,B}^{\Gw'}(\cdot,y) =\gd_y$ in a generalized sense, where $\gd_y$ is the Dirac measure centered at $y\in \Gw'$. Therefore, it is natural to extend the meaning of solutions of \eqref{eq:Rob_func} by introducing the  notion of distributional solutions to our boundary value problem.
\begin{defi}
	\em{
		Let Assumptions~\ref{assump1} hold in a bounded Lipschitz domain $\Gw$.
		Let $\eta$ be a compactly supported measure on $\Gw$ with bounded variation.
		We say that $u\in W^{1,p'}_{\pwd}(\Gw)$ is a {\em distributional solution} of  the boundary value problem
		$(P,B)u=\eta$ in $\Gw$ if for any (continuous) $\phi\in H^1_{\partial \Gw_{\mathrm{Dir}}}(\Gw)$ satisfying 
		\begin{equation}\label{eq:Rob_func_adjoint}
		\begin{cases}
		P^*\phi=g_0+\diver\mathbf{g} & \text{~in~} \Gw, \\
		B^*\phi=-\mathbf{g}\cdot \vec{n} & \text{~on~}   \partial \Gw_{\mathrm{Rob}},
		\end{cases}
		\end{equation}
				with $g_0,\mathbf{g}^i\in \mathcal{D}(\Gw,\partial \Gw_{\mathrm{Dir}})$ for all $1\leq i \leq n$, we have
		\begin{equation}\label{st_est}
		\int_{\Gw}(ug_0-\nabla u\cdot \mathbf{g}) \dx=\int_{\Gw}\phi \deta.
		\end{equation}
	}
\end{defi}
 \begin{rem}
 	\em{
 		Note that
 	 \eqref{st_est} reads as
 	$\Upsilon(u)\!=\!\int_{\Gw}\!\phi \deta$, where 
 	$\Upsilon\!=\!g_0+\diver\mathbf{g} \!\in\! \big(W^{1,p'}_{\partial \Gw'_{\mathrm{Dir}}}(\Gw')\big)^*$.	
 	}
 \end{rem}
\begin{lem}\label{Greendelta}
	   	 Let $x_0\in \Gw'\Subset_R \Gw$.
	   	 For $\lambda<\lambda_c(P,B,\Gw')$, the function $G_{P-\lambda ,B}^{\Gw'}(x_0,\cdot)$ is a positive  distributional solution of the problem    
		$(P^*-\lambda,B^*)u=\delta_{x_0}$ in $\Gw'$. 
\end{lem}
\begin{proof}	
		By Lemma~\ref{bound grad2}, $G^{\Gw'}_{P-\lambda,B}(x,\cdot)\in W^{1,p'}_{\partial \Gw'_{\Dir}}(\Gw')$.		
		Let $\phi\in H^{1}_{\partial \Gw'_{\mathrm{Dir}}}(\Gw')$ be a weak solution of  \eqref{eq:Rob_func}.
		We need to show that for any $x_0\in \Gw'$ we have
		\begin{align*}
		&
		\int_{\Gw'}\big [G_{P-\lambda ,B}^{\Gw'}(x_0,y)g_0(y)-\nabla G_{P-\lambda ,B}^{\Gw'}(x_0,y)\mathbf{g}(y) \big ]\! \dy=\phi(x_0).
		\end{align*}
	 Recall that by \eqref{lateG1}
		\begin{align*}
		\int_{\Gw'}\big [G_{P-\lambda ,B}^{\Gw'}(x_0,y)g_0(y)-\nabla G_{P-\lambda ,B}^{\Gw'}(x_0,y)\mathbf{g}(y)\big]\! \dy=u(x_0),
		\end{align*}	
		where   $u\in H^{1}_{\partial \Gw'_{\mathrm{Dir}}}(\Gw')$ is a weak solution to the problem 
		\eqref{eq:Rob_func}. 
		
		Since $\lambda<\lambda_c(P,B,\Gw')$, the generalized maximum principle for $(P-\lambda,B)$ in $\Gw'$ implies that the weak solution $\phi$ to problem \eqref{eq:Rob_func} is unique.  Hence, $\phi(x_0) =u(x_0)$. 
		Moreover, if  $\Upsilon=g_0\gneqq 0$, and $x \in \Gw'$, then   
		$$\int_{\Gw'}G_{P-\lambda ,B}^{\Gw'}(x,y)g_0(y) \dy=\phi(x) =u(x)>0.$$
		Since, $x$ and $g_0$ are arbitrary, we deduce that $G_{P-\lambda ,B}^{\Gw'}(x,\cdot)$ is positive a.e. in $\Gw'$.
\end{proof}
\begin{lemma}\label{uniqdist}
	Assume that $G_1$ and $G_2$ are distributional solutions of the problem $(P^*-\lambda,B^*)=\delta_{x_0}$ in  a bounded Lipschitz domain $\Gw'\Subset_R \wb$, where $x_0\in \Gw'$ and $\lambda<\lambda_c(\Gw')$. Then $G_1=G_2$ in $\Gw'$.  
\end{lemma}
\begin{proof}
	According to our assumptions the function $G:=G_1-G_2$ is a distributional solutions to the problem
	$(P^*-\lambda,B^*)=0$ in $\Gw'$. 
	Recall that for any $\psi\in C_0^{\infty}(\Gw')$, there exists a unique solution to the problem
	$(P-\lambda,B)\phi=\psi$ in $\Gw'$, and therefore,
	$$
	\int_{\Gw'}G(x,y)\psi(y)\dx=0.
	$$
	Since $\psi$ is arbitrary $G=0$.
\end{proof}
We proceed with regularity properties of distributional solutions near $\pwr$.
Recall that for $x\in \R^n$, we use also the notation $x=(x',x_n).$ 
\begin{lem}\label{reg_dist}
	Let Assumptions~\ref{assump2} hold in $\Gw$, and let $\Gw'\Subset_R \Gw$ be a Lipschitz subdomain. Assume that $(P,B)\geq 0$ in $\Gw$, and let $\eta$ be a compactly supported measure on $\Gw$ with bounded variation such that $\emptyset\neq  \mathrm{supp}(\eta)  \Subset \Gw$.
	Finally, let $v$ be a positive distributional solution of the problem
	$(P^*,B^*)v=\eta$ in $\Gw'$. 
	Then $v\in \mathcal{H}_{P^*,B^*}^{0}(\Gw'\setminus  \mathrm{supp}(\eta))$. 
\end{lem}
\begin{proof}
	By the ground state transform, we may assume that 
	$(P,B){\bf 1}=0$. In other words, we may assume that the coefficients of $(P,B)$ satisfy $\gamma=c=\bt=0$ (see \eqref{eq-gs}).
	
	Assume first that
	$$\Gw'=B_1(0)^+=\{x\in \R^n \mid x_n>0\text{~and~}|x|<1 \}, \quad \partial \Gw'_{\mathrm{Rob}}=\{x\in \R^n \mid   x_n=0\text{~and~}|x|<1\}.$$
	Denote $B_1(0)^-=\{x\in \R^n \mid  x_n<0\text{~and~}|x|<1 \}$,
	and $\mathrm{supp}(\eta)\cap \Gw'=\emptyset$.
	Consider the operator $(\breve{P}^*,\breve{B}^*)$, the `even' extension of $(P^*,B^*)$ into $B_1(0)$, namely,
	\begin{equation*}
		\breve{a}^{ij}(x',x_n) :=
		\begin{cases}
		{a}^{ij}(x',x_n) & x_n\geq 0,\\[2mm]
		{a}^{ij}(x',-x_n) & x_n< 0,
		\end{cases}
	\qquad
	\breve{\bar{\mathbf{b}}}(x',x_n)^j :=
	\begin{cases}
	\bb^j(x',x_n) & x_n\geq 0, 1\leq j\leq n,\\[2mm]
	{\bb}^j(x',-x_n)& x_n< 0, 1\leq j \leq n.
	\end{cases}
	\end{equation*}
	Similarly, define the `even' extension of the given solution $v$  into $B_1(0)$ by
\begin{equation*}
\breve{v}(x',x_n):=
\begin{cases}
v(x',x_n) & x_n\geq 0,\\[2mm]
v(x',-x_n) & x_n< 0.
\end{cases}
\end{equation*}	
	We claim that $\breve{v}\in L^1(B_1(0))$ is a distributional  solution to   the Dirichlet problem $\breve{P}^*(u)=0$ in $B_{1}(0)$ in the following sense (see, \cite[Definition~9.1]{ST}):
	$$
	\int_{B_{1}(0)}\breve{v} \psi \dx=  \int_{B_{1}(0)}\breve{v}\breve{P} \phi \dx=0,
	$$
	for all {\em admissible} $\phi$, i.e., $\phi\in H^1_{0}(B_{1}(0))\cap C(\overline{B_{1}(0)})$ and $\breve{P}(\phi)=\psi\in C(\overline{B_{1}(0)})$.
	
	Since  ${v}\in W^{1,p'}(B_1(0)^+)$, it follows that $\breve{v}\in L^1(B_1(0))$.
	Let us fix an admissible function $\phi$,
	and consider the functions 
	$$
	\phi_1=\phi \, \chi_{B_1(0)^+},\quad
	\psi_1=\psi \, \chi_{B_1(0)^+},\quad
	\phi_2=\phi \, \chi_{B_1(0)^-},\quad
	\psi_2=\psi\, \chi_{B_1(0)^-}.
	$$
	Then 
	$$
	\int_{B_1(0)}\!\!\breve{v}\psi\dx\!=\!
	\int_{B_1(0)^+}\!\!v\psi_1\dx
	+\!
	\int_{B_1(0)^-}\!\!\breve{v}\psi_2\dx
	\!=\!\!\int_{B_1(0)^+}v\psi_1\dx
	+
	\!\int_{B_1(0)^+}\!\!v(x',x_n)\psi_2(x',\!-x_n)\dx.
	$$
It remains to show that $$(P,B)(\phi_1+\phi_2(x',-x_n))=\psi_1+\psi_2(x',-x_n) \qquad \text{~in~} B_1(0)^+.$$ 
	Indeed, let $\vgf\in \mathcal{D}(B_1(0)^+,\partial B_1(0)^+_{\mathrm{Dir}})$ and let 
	$$
	\vartheta(x):=
	\begin{cases}
	\vgf(x',x_n) & \text{~in~} B_1(0)^+, \\
	\vgf(x',-x_n) & \text{~in~} B_1(0)^-. 
	\end{cases}
	$$
	Then $\vartheta\in H^1_{0}(B_1(0))$.
 Note that for any $\psi\in C_0^{\infty}(B_1(0))$,
	$$
	\int_{B_1(0)}D_n\vartheta \psi \dx+\int_{B_1(0)}\vartheta D_n\psi \dx=\int_{\partial B_1(0)}\phi(x',0)\vartheta-\phi(x',0)\vartheta \, \text{d}\sigma=0.
	$$
	Next, we compute
	\begin{align*}
	&
	\int_{B_1(0)^+}\left [A\nabla (\phi_1+\phi_2(x',-x_n))\cdot \nabla \vgf +\bb\cdot \nabla (\phi_1+\phi_2(x',-x_n))\vgf \right] \!\dx= \\ & 
	\int_{B_1(0)^+}\!\!\!\big [\breve{A}\nabla \phi \!\cdot\! \nabla \vartheta \!+\!\breve{\bb} \!\cdot\! \nabla \phi\vartheta \!\big ]\!\! \dx \!+ \!\!
	\int_{B_1(0)^-}\!\!\!\big[\!\breve{A}\nabla \phi(x',x_n) \!\cdot\! \nabla \vartheta \!+\! \breve{\bar{\mathbf{b}}} \!\cdot\!  \nabla \phi(x',x_n)\vartheta \!\big] \!\dx \!= \!\!
	\int_{B_1(0)}\!\!\psi \vartheta \!\dx.
	\end{align*}
	On the other hand,
	\begin{align*}
	\int_{B_1(0)^+}\!\![\vgf \psi_1+\vgf\psi_2(x',-x_n)]\!\dx \!=\!\! 
	\int_{B_1(0)^+}\!\!\vgf\psi_1\!\dx \!+ \!\!\int_{B_1(0)^-}\vgf(x',-x_n) \!\! \psi_2(x',x_n)\!\dx\!= \!\! \int_{B_1(0)}\!\!\!\vartheta \psi \!\dx.
	\end{align*}
	By a standard argument of local `flattening'  $\pwr$  (see for example, \cite[Appendix C.5]{EV}), we deduce that if $v$ is a distributional solution to the problem $(P^*,B^*)v=\eta$ in $\Gw'$,  then  for any $\Gw''\Subset_R \Gw'$ and any $y_0\in \partial \Gw''_{\Rob}$,
	$\breve{P}^*\breve{v}=0$ (in the distributional sense)
	in some ball $B_{\varepsilon}(y_0)$.
	Extend $\breve{v}$ to $\Gw'$ by letting $\breve{v}=v$ in $\Gw'$.
	As a consequence, $\breve{P}^*\breve{v}=\eta$ in $\mathrm{int}\left (\Gw''\cup (\overline{B_{\varepsilon}(y_0)})\right )$ in the distributional sense.
	 By \cite[Theorem 9.3]{ST} (see also \cite[Section 5]{K}), $\breve{v}\in H^1_{\loc}\left(\mathrm{int}\left (\Gw'\cup (\overline{B_{\varepsilon}(y_0)})\right )\setminus \mathrm{supp}(\mu)\right)$. 
	 
	 Let $\{v_l\}_{l \in \N}\subset \mathcal{D}(\partial \Gw'_{\Dir},\Gw')$ be a sequence which converges in $W^{1,p'}(\overline{\Gw'}\setminus \partial \Gw'_{\Dir})$ to $v$. 
	 By  the definition of weak solution, for any $\phi\in \mathcal{D}(\partial \Gw'_{\Dir},\Gw')$ supported outside $\mathrm{supp}(\eta)$, $(P,B)\phi=\Upsilon $ in $\Gw'$
	 implying
	 \begin{align}
	 \Upsilon(v_l)=\int_{\Gw'}A\nabla \phi \nabla v_l+\bb\nabla \phi v_l\dx.
	 \end{align}
	 Letting $l\to \infty$ and the definition of distributional solution imply 
	 $$
	 0=\Upsilon(v)=\int_{\Gw'}A\nabla \phi \nabla v+\bb \nabla \phi v \dx,
	 $$
	 i.e., $(P^*,B^*)v=0$ in $\Gw'\setminus \mathrm{supp}(\eta) $  in the weak sense.
\end{proof}
As a result of Lemma~\ref{reg_dist}, we obtain the following interior regularity of $G_{P-\lambda ,B}^{\Gw'}$.
\begin{lemma}\label{lem_4-13}
	Let $\Gw' \Subset_R \wb$ be a  Lipschitz subdomain of $\Gw$, and let  $\gl \!< \!\lambda_0(P,B,1,\Gw)$.   
		
		Then for any $x_0\in\Gw''\Subset_R\Gw'$,  and $\varepsilon>0$  such that $B_{\vge}(x_0)\Subset \Gw'$, $G_{P-\lambda ,B}^{\Gw'}(x_0,\cdot)\in H^{1}(\Gw''\setminus B_{\varepsilon}(x_0))\cap W^{1,p'}_{\partial \Gw'_{\mathrm{Dir}}}(\Gw')$ is positive, and satisfies $(P^*,B^*)G_{P-\lambda ,B}^{\Gw'}(x_0,\cdot)=0$ in $\Gw'\setminus \overline{B_{\varepsilon}(x_0)}$ in the weak sense. 
\end{lemma}
\begin{lem}\label{lem_adj_positivity}
	Assume that $(P,B)\geq 0$ in $\Gw$ and let Assumptions~\ref{assump2} hold in $\Gw$. Then  $(P^*,B^*)\geq 0$ in $\Gw$.
\end{lem}
\begin{proof}  Let $\{\Gw_k \}_{k\in \N}$ be an exhaustion of $\wb$.  Choose a sequence $\{x_k\}_{k\in \N}\subset \Gw$  converging  either to a point $\xi$ that belongs to a Dirichlet (Lipschitz)-portion of $\partial \Gw$ if $\Gw$ is bounded, or otherwise, to infinity. Without loss of generality, we may assume that $x_k\in \Gw_k$. By lemmas~\ref{lem_4-13} and ~\ref{reg_dist},  the operator $(P,B)$ admits a positive Green function $G_{P,B}^{\Gw_k}(x_k,\cdot)$ with singularity at $x_k$  which solves weakly the equation $(P^*,B^*)u=0$ in $\Gw_k \setminus \{x_k\}$. Consider the sequence  
	$$u_k(y):=\frac{G_{P,B}^{\Gw_k}(x_k,y)}{G_{P,B}^{\Gw_k}(x_k,x_1)}\, .$$ 
	  By the Harnack convergence theorem (Lemma~\ref{HCP}), $u_{k}$ admits a subsequence converging to a positive solution of $(P^*,B^*)u=0$ in $\Gw$.  
\end{proof}
\begin{lem}
	Let $\Gw' \Subset_R \Gw$ be a Lipschitz subdomain and let $\eta$ be a Radon measure compactly supported on $B_{\varepsilon}(x_0)\Subset\Gw'$. Let $u\in W^{1,p'}_{{\partial \Gw'_{\Dir}}}(\Gw')$ be a positive distributional solution to the problem
	$(P,B)w=\eta$ in $\Gw'$.
	Then
 \begin{equation}\label{need1}
 \|u\|_{W^{1,p'}(\Gw')}\leq C(\varepsilon)\int_{\Gw'}\text{d}\eta,
 \end{equation}
		where $C(\varepsilon)$ does not depend on $u$.

If in addition, $u\in H^1(\Gw'\setminus \mathrm{supp}(\eta))$ is a weak solution to the problem $(P,B)w=0$ in $\Gw'\setminus \mathrm{supp}(\eta)$, then
for any Lipschitz subdomain $\Gw''\Subset_R \Gw' \setminus \mathrm{supp}(\eta)$ we have
\begin{equation}\label{need2}
\sup_{\Gw''}|u| \leq C(\Gw'',\varepsilon)\int_{\Gw'}\, \text{d}\eta.
\end{equation}	
	\end{lem}
\begin{proof}
	Let $g_0$
	and $\mathbf{g}$ satisfying $g_0,\mathbf{g}^i\in \mathcal{D}(\Gw',\partial \Gw'_{\mathrm{Dir}}), \forall 1\leq i \leq n$.
	Writing $\Upsilon=g_0+\diver\mathbf{g}\in (W^{1,p'}_{\partial \Gw'_{\mathrm{Dir}}}(\Gw'))^*$ we obtain
	$$
	u(\Upsilon)=\int_{\Gw'}(ug_0 -\nabla u \cdot \mathbf{g})\dx=\int_{B_{\varepsilon}(x_0)}\phi \,\text{d}\eta,
	$$
	where $\phi$ is a weak solution of \eqref{eq:Rob_func_adjoint} with $\Gw$ replaced by $\Gw'$.
	By Lemma~\ref{bound grad2}, 
	$$\left |\int_{B_{\varepsilon}(x_0)}\phi \,\text{d}\eta \right |\leq \|\Upsilon\|_{(W^{1,p'}(\Gw'))^*}C(\varepsilon)\int_{B_{\varepsilon}(x_0)}\,\text{d}\eta.$$
	By the denseness of smooth functions in $L^p$-spaces, we obtain that for any $\Upsilon\in (W^{1,p'}(\Gw'))^*$, 
	$$|u(\Upsilon)|\leq  \|\Upsilon\|_{(W^{1,p'}(\Gw'))^*}C(\varepsilon)\int_{B_{\varepsilon}(x_0)}\,\text{d}\eta,$$
implying that $\|u \|_{W^{1,p'}(\Gw')}\leq C(\varepsilon)
\int_{B_{\varepsilon}(x_0)}\,\text{d}\eta $.

	Let $q=p'$, and note that $1<q<2$.
	The Sobolev embedding theorem \cite[Theorem~5.8]{L} states that 
	$$\|u\|_{L^{q^*}(\Gw')}\leq C \|u\|_{W^{1,q}(\Gw')}\qquad \forall u\in W^{1,q}(\Gw'),$$ where 
	$q^*={nq}/(n-q)$ is the critical Sobolev exponent.
	
	Assume in addition that $u\in H^1(\Gw'\setminus \mathrm{supp}(\eta))$ is a weak solution to the problem $(P,B)u=0$ in $\Gw'\setminus \mathrm{supp}(\eta)$, then    \cite[Theorem~5.36]{L} implies that
	for any $\Gw''\Subset_R \Gw'  \setminus \mathrm{supp}(\eta)$,
	$$
	\sup_{\Gw''}|u|\leq C\|u\|_{L^{q^*}(\Gw')},    
	$$
	and by the first part of the proof we are done.
\end{proof} 
The following well known proposition is a consequence of the Lebesgue differentiation theorem.
\begin{proposition}\label{prop:l}
	Let $\Gw'\Subset_R \Gw$ be a Lipschitz subdomain containing $x$.  For all $k\in \N$, consider the functions $f_{k,x}:=|B_{1/k}(x)|^{-1}\chi_{B_{1/k}(x)}$.
	Then, for any $\phi\in H^1(\Gw')$
	\begin{equation}\label{lebeg_diff}
	\lim_{k\to \infty}\int_{\Gw'}\phi(y) f_{k,x}(y) \dy=\phi(x)\qquad \mbox{for a.e. } x\in \Gw'. 
	\end{equation}
	If $\phi$ is continuous in $\Gw'$, then  \eqref{lebeg_diff} holds for all $x\in \Gw'$.
\end{proposition}
As a corollary of Proposition~\ref{prop:l}, Lemma \ref{lem_adj_positivity}, and  estimates \eqref{need1}, and \eqref{need2}, we obtain the following approximation result.
\begin{cor}\label{Green_app_dist}
	Let $q\!=\!p'$. For all $k\! \in \! \N$, consider the functions $f_k\!:=\!|B_{1/k}(x_0)|^{-1}\chi_{B_{1/k}(x_0)}$, where $x_0\in \Gw'\Subset_R \Gw$ and $\Gw'$ is a Lipschitz bounded subdomain.
	Let $u_k\in H^{1}_{\partial \Gw'_{\mathrm{Dir}}}(\Gw')$ be a sequence of positive (H\"older continuous)  solutions to the problem
	$(P^*,B^*)u_k=f_k$ in $\Gw'$.
	Then,  $u_k$ converges weakly (up to a subsequence) in $W^{1,q}(\Gw')$ and locally uniformly in $\overline{\Gw'}\setminus (\partial \Gw'_{\mathrm{Dir}}\cup \{x_0\})$ to a positive distributional solution of the problem
	$(P^*,B^*)u=\delta_{x_0}$ in $\Gw'$, namely, to $G^{\Gw'}_{P,B}(x_0,\cdot)$. 
\end{cor}
\begin{rem}\label{rem514}
	\em{
		By Corollary~\ref{Green_app_dist} 
		$G_{P-\lambda,B}^{\Gw'}(x_0,\cdot)$ is H\"older continuous in $\Gw' \setminus \{x_0\}$ and satisfies  weakly $(P^*,B^*)u=0$ in $\Gw'\setminus \{x_0\}$. Moreover, Corollary~\ref{Green_app_dist} and \cite[Theorem~8.29]{GT} imply that if $\mathcal{T}$ is a Lipschitz-portion of $\partial \Gw'_{\mathrm{Dir}}$, then $G^{\Gw'}_{P,B}(x_0,\cdot)$ vanishes continuously on $\mathcal{T}$. 
		 We remark that this conclusion can be also deduced  from the  approximation method in \cite[Theorem~9.2]{ST} in which one obtains that $G_{P-\lambda,B}^{\Gw'}(x_0,\cdot)$ is a (locally uniformly) limit of positive solutions $\{u_k\}_{k\in \N}$ of 
		$((P^k)^*,B^*)u_k=0$ in $\Gw'\setminus B_{1/k}(x_0)$.
		Here $(P^k)^*$ is the mollification of the  operator $P^*$ whose  coefficients are smooth in $\R^n$. 
	}
\end{rem}
%
\begin{rem}\label{G_D} {\em 
	Suppose that $(P,B)$ admits a positive Green function $G=G^{\Gw}_{P,B}(x_0,\cdot)$ in a bounded Lipschitz domain $\Gw$. 	Let $v$ be a positive solution of  the Dirichlet problem 
			$P^*v=0$ in $B_{\varepsilon}(x_0)  \Subset \Gw$ satisfying 
			$v=G$ on $\partial B_{\varepsilon}(x_0)$.
			Then $G-v$ is a distributional solution (in the Dirichlet sense \cite[Definition~9.1]{ST}) of the Dirichlet problem $P^*u=\delta_{x_0}$ in $B_{\varepsilon}(x_0)$, $u=0$ on $\partial B_{\varepsilon}(x_0)$. Therefore,  by the uniqueness of the Dirichlet Green function, $ G^{\Gw}_{P,B}(x_0,x)  =v(x)+G_{P}^{B_{\varepsilon}(x_0)}(x_0,x)$ in $B_{\varepsilon}(x_0)$, where 
			$G_{P}^{B_{\varepsilon}(x_0)}(x_0,x)$ is the  Dirichlet Green function in $B_{\varepsilon}(x_0)$ with a pole at $x_0$ (see also  \cite[Theorems 1 and 5]{Serrin}).
	}
\end{rem}
As a consequence of the previous subsection we obtain the following important result.  
\begin{lem}\label{G(x,y)}
	Let $\Gw'\Subset_R \Gw$ be a Lipschitz subdomain.
	Then $G_{P,B}^{\Gw'}(x,y)=G^{\Gw'}_{P^*,B^*}(y,x)$ for all $x,y\in \Gw'$, $x\neq y$.
\end{lem}
\begin{proof}
	Fix $x_0\neq y_0$ in $\Gw'$, and let $u_k$ (resp., $v_k$) be an approximating sequence of $G^{\Gw'}_{P,B}(x_0,\cdot)$ (resp., $G^{\Gw'}_{P^*,B^*}(y_0,\cdot)$) with $g_k$ (resp., $f_k$) as in Corollary~\ref{Green_app_dist}, where $g_k$ is an approximating sequence of $\gd_{y_0}$ (resp., $f_k$ is an approximating sequence of $\gd_{x_0}$). 
	For all $k$,
	\begin{align*}
	\mathcal B_{P^*,B^*}(u_k, v_k) =\int_{\Gw'} f_k v_k \dx\to G^{\Gw'}_{P,B}(y_0,x_0) \quad \mbox{as~} k\to \infty.
	\end{align*}
	On the other hand, 
	\begin{align*}
	\mathcal{B}_{P^*,B^*}(u_k, v_k) =\mathcal B_{P,B}(v_k, u_k)=
	\int_{\Gw'}g_ku_k\dx \to G^{\Gw'}_{P^*,B^*}(x_0,y_0) \quad \mbox{as~} k\to \infty. \qquad \qedhere
	\end{align*}	 
\end{proof}
\subsection{Green function for $(P,B)$ satisfying Assumptions~\ref{assump2}} 
We proceed with the construction of the positive minimal Green function of $(P,B)$ in a domain $\Gw$. 
\begin{theorem} \label{thm_Gr}
Let $(P,B)\geq 0$ in $\Gw$  and satisfies  Assumptions~\ref{assump2}.
Then either $(P,B)$ admits a positive minimal Green function in $\Gw$, or else,  $(P,B)$ admits a ground state. 
\end{theorem} 
\begin{proof}
Let $\{\Gw_k \}_{k\in \N}$ be an exhaustion of $\wb$. Recall that for any $\Gw' \Subset_R \wb$ the generalized maximum principle holds in $\Gw'$,  hence, $(P,B)$ admits a positive Green function $G_{P ,B}^{\Gw'}$ in $\Gw'$. 
Let $x_0,x_1\in \Gw$ with $x_0\neq x_1$.
For each $k\in \N$, consider the function $G_k(x,y)=G_{P ,B}^{\Gw_k}(x,y)$. 

We claim that the sequence $\{G_k(x,y)\}_{k\in \N}$ is monotone nondecreasing. 
Indeed, for any $0\lneqq \phi\in \mathcal{D}(\Gw_k,\partial \Gw_{k,\Dir})$, and $x\in \Gw_k$,
$$
u_{k+1}(x)-u_k(x) \! =\!\!\int_{\Gw_{k+1}}\!\!\!\! G_{k+1}(x,y)(x,y))\phi \dy - 
\!\!\int_{\Gw_k} \!\!\! G_{k}(x,y))\phi \dy
	\!=\!\! \int_{\Gw_k}\!\!\!(G_{k+1}(x,y)-G_{k}(x,y))\phi \dy,
$$ 
where $u_{k+1}$ and $u_k$ are positive solutions to the equation $(P,B)u=\phi$ in $\Gw_k$ and $\Gw_{k+1}$ respectively. Moreover, $u_k< u_{k+1}$ on $\partial\Gw_{k,\Dir}$ implying that $u_k < u_{k+1}$ in  $\Gw_k$. Since $\phi$ is arbitrary, the monotonicity of $G_{k}(x,y)$ follows.  

Next we show  the following dichotomy:

{\bf Case 1:} Assume first that for some $y_0\neq x_0 \in \Gw$, the sequence $\{ G_k(x_0,y_0)\}_{k\in \N}$ is  bounded, then by the monotonicity and Harnack convergence principle, 
 the sequence $\{G_k(x,y)\}_{k\in \N}$ converges locally uniformly in $\Gw\setminus \{y\}$ to a positive solution  of $(P,B)u=0$ in $\Gw \setminus \{ y\}$, denoted by $G^{\Gw}_{P,B}(x,y)$.
Consider again the function 
$u_k(x)\! = \!\int_{\Gw_k}G_k(x,y)\phi(y) \!\dy$
which satisfies
$(P,B)u_k=\phi$ in $\Gw_k$.
Then, for any $0\lneqq \phi\in C_0^\infty(\Gw)$, the monotone convergence theorem implies that 
$$
\lim_{k\to \infty} u_k(x)=\lim_{k\to \infty}\int_{\Gw_k}G_k(x,y)\phi(y)\dy=
\int_{\Gw}G^{\Gw}_{P,B}(x,y)\phi(y)\dy.
$$
Further, for  $x_1\notin \mathrm{supp}(\phi)$ the sequence $\{u_k(x_1)\}_{k\in \N}$ is bounded by 
$
\int_{\Gw}G_{P,B}^\Gw(x_1,y)\phi(y) \dy.
$
Therefore,
  $u_k\to u$ locally uniformly in $\Gw$ and $(P,B)u=\phi \gneqq 0$ in $\Gw$. In particular, $u$ is a regular positive supersolution of  $(P,B)$, and  therefore, $(P,B)$ is subcritical in $\Gw$. 
   Moreover, $(P,B)G_{P,B}^\Gw(\cdot,y)=\gd_y$ in $\Gw$ in the distributional sense, and $(P,B)G_{P,B}^\Gw(\cdot,y)=0$ in $\Gw\setminus \{y\}$ in the weak sense. 
  Clearly, the uniqueness and minimality of such a  Green function follows from a standard comparison argument. We call $G^{\Gw}_{P,B}$  the {\em positive minimal Green function of the operator $(P,B)$ in $\Gw$}. 

	{\bf Case 2: }
	Assume that the sequence $\{ G_k(x_0,y_0)\}_{k\in \N}$ is  unbounded, and again  by the monotonicity and Harnack convergence principle, 
	the sequence $\{G_k(x,y)\}_{k\in \N}$ converges locally uniformly in $\Gw\setminus \{y\}$ to
	$\infty$. 
	Let 
	$$
	u_k(x)=\int_{\Gw_k}G_k(x,y)\phi(y) \dy\in H^1_{\partial \Gw_{k,\Dir}}(\Gw_k).
	$$
	Fix $x_1\in \Gw_1$. By the monotone convergence theorem, $$\lim_{k\to \infty}u_k(x_1)=\lim_{k\to \infty}\int_{\Gw_k}G_k(x_1,y)\phi(y)\dy=\infty.$$
	Therefore, the sequence $\{ \vgf_k:=u_k/u_k(x_1)\}_{k\in \N}$ converges locally uniformly to a positive solution $\vgf$ of the equation $(P,B)v=0$ in $\Gw$. 
	
	We claim that the function $\vgf$ is a ground state.
	Indeed,
	let $K\Subset_R\Gw$ be a Lipschitz subdomain such that $\mathrm{supp}(\phi)\Subset K$.
	Let $v$ be a positive continuous supersolution
	of the equation $(P,B)w=0$ in $\Gw\setminus K$ such that $\vgf\leq Cv$ in 
	$\partial K_{\Dir}$. 
	Recall that $\vgf_k\to \vgf$ uniformly on $K$ and therefore for any $\varepsilon>0$ there exists 
	$\vgf_k$ satisfying $\vgf_k\leq (C+\varepsilon) v$ on $\partial K_{\Dir}$.
	By the generalized maximum principle in ${\Gw_k}\setminus  K$, we have $\vgf_k\leq (C+\varepsilon)v $ in $\Gw_k\setminus  K$. Letting $k\to \infty$ we deduce that 
	$\vgf\leq (C_1+\varepsilon)v$ in $\Gw \setminus K$. Since $\vge>0$ is arbitrarily small we have,
	$\vgf\leq Cv$ in $\Gw \setminus K$, implying that $\vgf$ has minimal growth. By the uniqueness of the ground state, it follows that $\vgf$ does not depend on $\phi$. 
	\end{proof}
As a corollary of Lemma~\ref{G(x,y)} and Theorem~\ref{thm_Gr}, we have: 
\begin{cor}
	Assume that $(P,B)\geq 0$ in $\Gw$. Then 
	$(P,B)$ is subcritical in $\Gw$ if and only if 
	$(P^*,B^*)$ is subcritical in $\Gw$.
\end{cor}
The approximation argument in Corollary~\ref{Green_app_dist} and Harnack convergence principle readily imply the following result.
	\begin{lem}\label{Green_app_dist2}
		Assume that $(P,B)$ is subcritical in $\Gw$.
		Let $\{\Gw_k\}_{k\in \N}$ be an exhaustion of  $\wb$, and
		let $q=p'$. Consider the functions $f_k:=|B_{1/k}(x_0)|^{-1}\chi_{B_{1/k}(x_0)}$ where $x_0\in \Gw_1$. 
		Let $u_k\in H^{1}_{\partial \Gw_{k,\Dir}}(\Gw_k)$ be a sequence of positive (H\"older continuous)  solutions to the problem
		$(P,B)u_k=f_k$ in $\Gw_k$.
		Then,  $u_k$ converges  locally uniformly in $\wb$ to   $G^{\Gw}_{P,B}(\cdot,x_0)$. In particular, $G^{\Gw}_{P,B}(\cdot,x_0)\in \mathcal{H}^{0}_{P,B}(\Gw \setminus \{ x_0\})$ is a positive solution in  $\Gw \setminus \{ x_0\}$ of minimal growth in a neighborhood of infinity in $\Gw$.
	\end{lem}
\begin{rem}\label{rem:green}
	\em{
	Assume that $(P,B)$ is subcritical in $\Gw$, and let $u\in \mathcal{H}^{0}_{P,B}(\Gw)$. 
	 Consider the corresponding ground state transform $(P^u,B^u)$ in $\Gw$. 
	Then in $L^2_\loc(\Gw)$
	$$G_{P^u,B^u}^{\Gw}(x,y)=\frac{G_{P,B}^{\Gw}(x,y)u(y)}{u(x)}\,. $$
	}
\end{rem}
We conclude the present section with some basic properties of critical and subcritical operators $(P,B)$.	
\begin{rem}
		\em{
Using the same proofs as in  \cite{PS} and references therein, one deduce the following assertions for $(P,B)\geq 0$ in $\Gw$: 
 \begin{enumerate}
 	\item  Let $V\in L^{p/2}_{\loc}(\wb)$, $p>n$. Then $$S:=\{\lambda \in \R:\mathcal{H}^{0}_{P-\lambda V,B}\neq \emptyset \}$$ is a closed  interval, and if $\gl \in \mathrm{int}(S)$, then $(P-\gl,B)$ is subcritical in $\Gw$. Furthermore, $S$ is unbounded if and only if $V$ does not change its sign in $\Gw$. 
 	Moreover, if $V$  has a compact support in $\Gw$, and $(P,B)$ is subcritical in $\Gw$, then $\mathrm{int}(S)\neq \emptyset$ and $\gl\in \partial S$ if and only if  $(P-\gl,B)$ is critical in $\Gw$.  
		\item {Protter-Weinberger formula \cite[Theorem~7.12]{PS}}: Let $V\in L^{p/2}_{\loc}(\wb)$, $p>n$, be a positive potential, and consider the sets:
		$$K:=\{0<u\in H^1_{\loc}(\wb) \}, \qquad M:= \{\phi\gneqq 0 \mid \phi \in \mathcal{D}(\Gw,\pwd)\}.$$
		Then
		$$
		\lambda_0(P,B,V,\Gw)=\sup\limits_{u\in K}
		\inf\limits_{\phi\in M}\dfrac{\mathcal{B}_{P,B}(u,\phi)}{\int_{\Gw} Vu \phi \dx}\,.
		$$
		\item  Let $ V_1,V_2\in L^{p/2}_{\loc}(\wb)$, $V_1\leq V_2$. Then $\lambda_0(P,B,V_2,\Gw)\leq \lambda_0(P,B,V_1,\Gw)$. But if $\mathcal{SH}_{P,B}(\Gw)= \emptyset$ and $V_1\geq 0$, then 
		$\lambda_0(P,B,V_2,\Gw)\leq \lambda_0(P,B,V_1,\Gw)$ \cite[Lemma~7.10]{PS}.
		\item  If $ W,V_1,V_2\in L^{p/2}_{\loc}(\wb)$ such that $V_1\leq V_2$
		then $\lambda_0(P+V_2,B,W,\Gw)\geq \lambda_0(P+V_1,B,W,\Gw)$. Moreover,  $\lambda_0(P,B,V,\Gw)$ is a concave function of $V$, \cite[Lemma~7.9]{PS}.
		\item Assume that $\Gw_1\Subset_R \Gw_2$ and assume that $(P,B)$ is subcritical in $\Gw_2$. Let $V_1,V_2\in L^{p/2}(\Gw_2) $ satisfying $0\leq V_1\leq V_2$. Then
		$G_{P+V_2,B}^{\Gw_1}(\cdot,y)\leq G_{P+V_1,B}^{\Gw_2}(\cdot,y)$ in $\Gw_1$ \cite[Corollary~8.22]{PS}.
	\end{enumerate}
}	
\end{rem}

\section{Symmetric operators}\label{sec_Symm}
Throughout this section {\bf we assume that  $(P,B)$ satisfies Assumptions~\ref{assump2} in $\Gw$, and that $(P,B)$ is {\em symmetric}},  in other words, we assume that $\bb=\bt$. We note that if $P$ is symmetric and $\bt$ is smooth enough, then $P$ is in fact a Schr\"{o}dinger-type operator of the form
\begin{equation*} \label{div_P7}
Pu:=-\diver \big(A(x)\nabla u\big) 
+\big(c(x)-\diver \bt(x)\big) u \qquad x\in\Gw.
\end{equation*}
We prove the appropriate Allegretto-Piepenbrink-type theorem for \eqref{P,B} and characterize criticality via the existence of a null-sequence. 
\subsection{Allegretto-Piepenbrink  theorem } This theorem is well known if $\pwr\!\!=\!\emptyset$, namely, when $\tilde P$ is the Friedrichs extension of $P$ (see, \cite{Agmon,P3,S} and references therein). We prove:
\begin{thm}[Allegretto-Piepenbrink-type theorem]
	Let Assumptions~\ref{assump2} hold in $\Gw$.
	Then a symmetric operator $(P,B)$ is nonnegative in $\Gw$ if and only if $\mathcal{B}_{P,B}(\phi,\phi)\!\geq\! 0$ for all $\phi \!\in\! \mathcal{D}(\Gw,\pwd)$. In other words, in the symmetric case we have  $\gl_0(P,B,1,\Gw)\!=\!\Gl(P,B,\Gw)$.
\end{thm}
\begin{proof}
	Fix an exhaustion $\{\Gw_k\}_{k\in \N}$ of $\wb$. If $\mathcal{B}_{P,B}(\phi,\phi)\geq 0$ for all $\phi\in \mathcal{D}(\Gw,\pwd)$, then for any $k>0$ the form $\mathcal{B}_{P+1/k,B}$ is coercive in $H^1_{\partial \Gw_{k,\Dir}}(\Gw_k)$, and the corresponding resolvent is a positive operator. 
	
	 For $k\in \N$, let  $f_k\in C_0^\infty(\Gw_{k}\setminus \Gw_{k-1})$ be a nonzero nonnegative function. 
	By Lemma~\ref{lem1.8}, there exists a unique   positive solution $v_k$ to the problem	
	$$
	\begin{cases}
	\left(P+\frac{1}{k} \right)w=f_k & \Omega_k, \\
	Bw=0 & \Pw_{k,\mathrm{Rob}},\\
	\mathrm{Trace}(w)=0& \partial {\Gw_{k,\Dir}}.
	\end{cases}
	$$
	Fix $x_1\in \Gw\setminus \Gw_1$, and consider the sequence $\{u_k:= v_k/v_k(x_1) \}_{k\in \N}$. 
	By the Harnack convergence principle  (Lemma~\ref{HCP}), there exists a subsequence of $\{u_k \}$  converging to a positive solution $u\in H^{1}_{\loc}(\Gw)$ of $(P,B)u=0$ in $\Gw$. Hence,
	$\gl_0(P,B,1,\Gw)\geq \Gl(P,B,\Gw)$. 
	
	\medskip
	
	Assume now that $(P,B)\geq 0$ in $\Gw$, and let $u\in \mathcal{H}_{P,B}^{0}(\Gw)$. 
	By the ground-state transform (Definition~\ref{gs+transform})  in subdomains $\Gw'\Subset_R \Gw$, we 
	obtain 
	$$
	\mathcal B_{P,B}(u\phi,u\phi) = \mathcal B_{P^u,B^u}(\phi,\phi)= \int_{\Gw}(a^{ij} D_j \phi D_i \phi )u^2\!\dx\geq 0  \qquad \forall u\phi \in \mathcal{D}(\Omega,\Pw_{\mathrm{Dir}}).
	$$
	Hence, $\mathcal B_{P,B}$ is nonnegative on $ \mathcal{D}(\Gw,\pwd)$. Hence, $\gl_0(P,B,1,\Gw)\leq \Gl(P,B,\Gw)$.
\end{proof}
\begin{rem}
	\em{
		If $(P,B)$ is symmetric in a $\Gw$ and Assumptions~\ref{assump1} hold in $\Gw$, then the above proof and Theorem~\ref{agmonmethod} imply that 
		$$\gl_0(P,B,1,\Gw)=\Gl(P,B,\Gw)=\Gg(P,B,\Gw)=\gl_c.$$
	}	
\end{rem}
 \subsection{Null-sequence and criticality}
\begin{defi}\em{
		A sequence $\{u_k \}_{k\in \N}\subset \mathcal{D}(\Gw,\pwd)$ is called a {\em null-sequence} with respect to $(P,B)$ in $\Gw$ if the sequence satisfies the following three properties:
		\begin{enumerate}
			\item $u_k \geq 0$ for all $k\in \N$,
			\item there exists a fixed open set $O\Subset \Gw$ such that $\|u_k \|_{L^2(O)}=1$ for all $k\in \N$,
			\item $\lim_{k\to \infty} \mathcal{B}_{P,B}(u_k,u_k)=0$.
		\end{enumerate}
	}
\end{defi}
For the characterization of criticality by the existence of a null-sequence in the particular case $\pwr=\emptyset$, see, \cite[Theorem~2.7]{M1} and \cite[Theorem~1.4]{PT}. 
\begin{thm}
	Let Assumptions~\ref{assump2} hold in $\Gw$. Assume that $(P,B)$ is symmetric  and $(P,B)\geq 0$ in $\Gw$. Then $(P,B)$ admits a null-sequence in $\Gw$ if and only if $(P,B)$ is critical in $\Gw$. 
\end{thm}
\begin{proof}
Let  $u\in \mathcal{H}^0_{P,B}(\Omega)$, then by the ground-state transform we have
$$
\mathcal B_{P,B}(u\phi,u\phi) = \mathcal B_{P^u,B^u}(\phi,\phi)= \int_{\Gw}(a^{ij} D_j \phi D_i \phi )u^2\!\dx,
$$
where  $u\phi \in \mathcal{D}(\Omega,\Pw_{\mathrm{Dir}})$.	
	
	We claim that if $(P,B)$ is critical in $\Gw$, then for any nonempty open set $O \Subset \Gw$, we have  
\begin{equation}\label{eq_cb}
	C_O:=\inf\limits_{\underset{0\leq \phi \in \mathcal{D}(\Gw,\pwd)}{ \| \phi\|_{L^2(O)}=1}}\mathcal{B}_{P,B}(\phi,\phi)=0.
\end{equation}
	Indeed, pick $0\lneqq W \in C^{\infty}_0(O)$ such that 
	$0 \leq W \leq 1$. Then for all $0\leq \phi \in \mathcal{D}(\Gw,\pwd)$ with 
	$\| \phi\|_{L^2(O)}=1$ we have
	$$
	C_O\int_{\Gw}W\phi^2 \leq C_O\leq \mathcal{B}_{P,B}(\phi,\phi).
	$$
	Hence, the criticality of $(P,B)$ clearly implies that $C_O=0$,  and therefore, there exists a minimizing sequence for \eqref{eq_cb} which is obviously a null-sequence.

\medskip
	
	Now, let $\{\phi_k \}_{k\in \N}$ be a null-sequence of $(P,B)$ in $\Gw$,
	let $v\in \mathcal{SH}_{P,B}(\Gw)$, and
	let $w_k:={\phi_k}/{{v}}$. Recall that by the  weak Harnack inequality (Lemma~\ref{lem_whi}), $v>0$, and therefore,  the sequence $\{w_k\}_{k\in \N}$ is well defined. 
	 Since $\phi_k$ is a null-sequence, it follows that 
	for any $K\Subset_R \wb$ there exists $C_K>0$ such that
	\begin{equation*}\label{theinequality}
	C_K\!\!\int_{K}|\nabla w_k|^2v^2\!\dx\! \leq \!\!\int_{K}|\nabla w_k|_A^2v^2\!\dx \!\leq \!\!\int_{\Gw}|\nabla w_k|_A^2 v^2\!\dx\!\leq\! 
	\mathcal{B}_{P^v,B^v}(w_k,w_k)\! \leq \!\mathcal{B}_{P,B}(\phi_k,\phi_k)\! \to\! 0
	\end{equation*}
 as $k \to \infty$, where the above inequality is a consequence of \eqref{eq:gs_eq} and the  weak Harnack inequality (Lemma~\ref{lem_whi}).
 
	Consequently, $\nabla w_k \to 0$ in $L^2_{\mathrm{loc}}(\wb)$.
	Rellich-Kondrachov theorem (see for example \cite[Theorem 8.11]{LL}) implies that up to a subsequence $w_k\to C\geq 0$ in $H^{1}_{\loc}(\wb)$. Hence, up to a subsequence 
	$\phi_k\to C v$ pointwise in $\Gw$ and also in $L^2_{\loc}(\wb)$. Therefore, any two positive supersolutions in $\mathcal{SH}_{P,B}(\Gw)$ are equal up to a multiplicative constant. In view of Lemma~\ref{inview}, $(P,B)$ is critical in $\Gw$.
\end{proof}
\appendix
\section[ ]{}	\label{appendix1}
The appendix is devoted to a construction (under Assumptions~\ref{assump2}) of a Lipschitz exhaustion of $\wb$ (see, Definition~\ref{def:exhaustion}).
 It is well known (see for example, \cite[Proposition~8.2.1]{Daners}), that for any domain $\Gw\subset \R^n$  there is an exhaustion $\{\Gw_k\}_{k\in \N}$ of bounded smooth domains satisfying  
$\Gw_{k}\Subset \Gw_{{k+1}} \Subset \Gw$, and $\bigcup\limits_{k\in \N} \Gw_{k}=\Gw$.
 
In order to obtain an exhaustion of $\wb$, we construct an exhaustion of $\pwr$ and glue it through  appropriate `cylinders' to a smooth compact exhaustion of $\Gw$. The difficulty that arises in such a process comes from the fact that $\pwr$ might be unconnected and only $C^1$-smooth. Moreover, the `cylinders' should touch $\pwr$ in `good directions', i.e., directions with respect to which $\pwr$  can be locally represented as the graph of a continuous function.

Existence of such an exhaustion is known for the case, where $\partial \Gw= \pwr\in C^3$ \cite{DR}. In our paper we consider  the case $\pwr\subset \partial \Gw$, and $\pwr\in C^1_\loc$ which forces us to take a different approach. We remark that our construction is inspired by \cite{BZ}.

We begin with the following geometric preliminaries.   
\begin{defi}\label{def_gd}
	{\em	
Let $\Gw\subset \R^n$ be a domain of class $C^0$.  For a point $x_0 \in  \R^n$, we define a
{\em good direction at $x_0$, with respect to a ball $B(x_0, \gd)$, $\gd\! >\! 0$}, with $B(x_0, \gd) \! \cap \! \partial\Gw \neq \emptyset$,  to be a vector $\mathcal{N}\!=\!\mathcal{N}(x_0)\! \in \! S^{n-1}$ 
such that there is an orthonormal coordinate system $Y=(y',y_n)=(y_1,\ldots ,y_n)$ with origin at the point $x_0$,  so that $\mathcal{N}=e_n $ is the unit vector in the $y_n$ direction, 
together with a continuous function  $f:\R^{n-1}\to \R$ 
(depending on $x_0,\delta$ and $\mathcal{N}$), such that
$$
\Gw \cap B_{x_0}(\delta)= \{ y\in \R^n: f(y')<y_n, |y|<\delta\}.
$$
We say that $\mathcal{N}$ is a good direction at $x_0$ if it is a good direction with respect to some ball
$B(x_0, \gd)$ with $B(x_0, \gd) \cap \partial\Gw \neq \emptyset$.

If $x_0 \in \partial \Gw$, then a good direction $\mathcal{N}$ at $x_0$ is called a {\em pseudonormal} at $x_0$ (see \cite{BZ}).		
	}
\end{defi}
\begin{rem}\label{rem_good}
	\em{
		If $\Gw\in C^1$, then any good direction at $x_0\in \partial \Gw$ is never tangent to  $\partial \Gw$ at   $x_0$. Indeed, the normal direction at $x_0$ is given by the vector $(\nabla f,-1)$ which is not orthogonal to the vector $e_n=(0,\ldots,0,1)$.	
	}
\end{rem}
\begin{proposition}\label{goodcone}
	Let $\Gw\subset \R^n$ be a $C^1$-domain. Let $x_0\in \partial \Gw$ and let $v\in S^{n-1}\setminus T_{x_0}\partial \Gw$, where 
	$T_{x_0}\partial \Gw$ is the tangent space to $\partial \Gw$  at $x_0$.  Then $v$ is a good direction at $x_0$. 
\end{proposition}
\begin{proof}
	 Let us assume without loss of generality that 
	 $x_0=0$.
	 There exists a local coordinate system, $\delta >0$, and a $C^1$ function $\phi:\R^{n-1}\to \R$  such that 
	 $\phi(0)=0$, $B_{x_0}(\delta)\cap \Gw=\{x\in \R^n:\phi(x')<x_n,|x|<\delta\}$.
	 
		Assume first that $v\in \mathcal{H}_i$ where $1\leq i \leq n$ is fixed and 
	$$\mathcal{H}_i:=\{z\in \R^n\mid z=(0,\ldots, 0,z_i,0,\ldots,0,z_n), z_i,z_n\neq 0\}.$$
	Denote by $0< \theta<\pi$ the angle between $v$ and $e_n$.
	Let $\mathcal{O}$ be the rotation which maps $e_n$ to $v$ and fixes
	$e_j$ for all $j\neq n$ and $j\neq i$. Namely,  
	{\small
	$$
	\mathcal{O}_{n\times n}=
	\begin{pmatrix}
	\begin{matrix}
	1& 0&\dots& \dots&\dots&\dots &\dots& 0\\
	\vdots& \vdots&\vdots&\vdots&\vdots&\vdots&\vdots&\vdots&\\
	0 & \dots &\cos \theta& 0& \dots& \dots& \dots& \sin \theta\\
	0& \dots&\dots& 1&0&\dots &\dots& 0\\
	\vdots& \vdots&\vdots&\vdots&\vdots&\vdots&\vdots&\vdots&\\
	0 & \dots &-\sin \theta& 0& \dots& \dots& \dots& \cos \theta\\
	\end{matrix}
	\end{pmatrix}.
	$$}
	By the implicit function theorem, $\mathcal{O}\begin{pmatrix}x'\\\phi(x')\end{pmatrix}$ is  a graph of a $C^1$ function and $e_n$ is a good direction as long as 
	$$
	\cos \theta +  \frac{\partial \phi}{\partial x_i}(0) \sin \theta\neq 0.
	$$
	The latter condition is satisfied once $v\notin T_{x_0}\partial \Gw$.
	
	For a general $v\in S^{n-1}\setminus T_{x_0}\partial \Gw$, we can write $v=\sum_{i=1}^{n+1}\alpha_i u_i$ where 
	$u_i\in \mathcal{H}_j$ for some $j$, and $u_i$ is a good direction at $x_0$,
	$0\leq \ga_i\leq 1$, $u_{n+1}=e_n$ and $\sum_{i=1}^{n+1} \alpha_i=1$. 
	By \cite[Lemma~2.2]{BZ}, $v/|v|$ is  a good direction.
\end{proof}
\begin{proposition}\cite[Proposition~2.1]{BZ}\label{field_good}
	Let $\Gw\subset \R^n$ be a bounded, open set with boundary of class $C^0$. Then there exists a neighborhood  $U$ of $\partial \Gw$ and a smooth function $\vec{N}:U\to S^{n-1}$ so that for each $x_0\in U$ the unit vector $\vec{N}(x_0)$ is a good direction at $x_0$.
\end{proposition}
We conclude our preliminaries with the following lemma. 
\begin{lemma}\label{lem_Lip}
	Let $T,S\subset \R^n$ be  $C^1$ domains and let $H:=T\cap S$. 
	For $x\in \partial H$, denote by $\vec{N}_T(x)$ and $\vec{N}_S(x)$ the corresponding normal vector fields to $\partial T$ and $\partial S$, respectively. Assume that for all $x\in \partial H$,
	$\vec{N}_T(x)$ and $\vec{N}_T(x)$ are linearly independent. Then each connected component of $H$ is a Lipschitz domain.
\end{lemma}
\begin{proof}
	By \cite[Lemma~7.1]{BZ}, it is enough to find $n$ linearly independent good directions at each $x_0\in \partial H$. By Proposition~\ref{goodcone} at each $x_0\in \partial H$ we can find an exterior cone of good directions with respect for both $S$ and $T$. This cone  contains  
	$n$ linearly independent good directions.
\end{proof}

Let Assumptions~\ref{assump2} hold in  a domain $\Gw \subset \R^n$ with non-empty boundary. 
The {\em signed distance $d_\Gw$ to}  $\partial \Gw$ is given by 
$$
d_{\Gw}(x):=
\begin{cases}
\;\;\;\inf\limits_{y\in \partial \Gw}|x-y| & ~\text{if} ~ x\in \Gw, \\[2mm]
-\inf\limits_{y\in \partial \Gw}|x-y| & ~\text{if} ~ x\notin \Gw.
\end{cases}
$$

Let $\rho(x):\R^n\to \R$ be a {\em regularized signed distance to $\partial \Gw$} \cite[Proposition~3.1]{BZ}, namely, $\rho(x)\in C^{\infty}(\R^n\setminus \partial \Gw)\cap C^{0,1}(\R^n)$ satisfies the following properties:
\begin{enumerate}
	\item For all $x\in \R^n \setminus \partial \Gw$, \quad  $\dfrac{1}{2}\leq \dfrac{\rho(x)}{d_{\Gw}(x)}\leq 2$.
	\item If $\Gw$ is bounded with continuous boundary, then  there exists a relative neighborhood $U$ of $\partial \Gw$ in $\Gw$ such that $|\nabla \rho|\neq 0$ in $U\setminus \partial \Gw$. 
\end{enumerate}

For $\varepsilon>0$ small enough, let
$$
D_{\varepsilon}:=\{x\in \Gw: \rho(x)>\varepsilon\} \neq \emptyset.
$$ 
By Sard's theorem \cite[Theorem~6.8]{Lee}, for  almost  every $\varepsilon>0$, the open sets $D_{\varepsilon}$ are smooth.
For each $\delta>0$ let 
$$
D_{\varepsilon,\delta}:=
D_{\varepsilon} \cap B_{0}(1/\delta).
$$
By \cite[Lemma~1]{DR} and Lemma~\ref{lem_Lip}, there exists $c_0$ such that for a.e. $0<\varepsilon,\delta<c_0$, we have that  $D_{\varepsilon,\delta}$ is a Lipschitz set  having a finite number of connected components.
\begin{lem}\label{ballem}
	For each $k>0$ there exist $\varepsilon, \delta>0$ such that
	such that $D_{\varepsilon}$ is smooth,  and if $x\in E$, where $E$ is a connected component of $\partial D_{\varepsilon,\delta}$, and $|x|<k$, then $|\nabla \rho(x)|\neq 0$ for all $x\in E \setminus \partial \Gw$.
\end{lem}
\begin{proof}
	As mentioned before, Sard's theorem implies that there exists a sequence $\{\varepsilon_l\}_{l\in \N}$ satisfying  $0<\varepsilon_l\to 0$ as $l\to \infty$, and
	such that $D_{\varepsilon_l}$ is smooth.
	By \cite[(3.4)]{BZ}, a regularized signed distance  function $\rho(x)$ is given by the implicit equation
	$$
	G(x,\tau)=\int_{|z|< 1}d_{\Gw}\left (x-\frac{\tau}{2}z\right)\varphi(z) \,\text{d}z,
	$$
	where $\vgf$ is a smooth nonnegative function on $\R^n$ supported on the unit ball such that $\int_{|z|<1}\varphi(z) \,\text{d}z=1$, and 
	\begin{equation}\label{rhoApp}
	\rho(x)=G(x,\rho(x)).
	\end{equation}

	If $|x|\leq k$ then 
	 for $\tau,\delta$ sufficiently small  we have 
	$$
	d_{\Gw}\left (x-\frac{\tau}{2}z\right)=
	d_{D_{\varepsilon,\delta}}\left (x-\frac{\tau}{2}z\right).
	$$
	Hence, by \eqref{rhoApp} and the properties of $\rho(x)$,
	$\nabla\rho(x)\neq 0$.
\end{proof}
Next, we modify  the exhaustion $\{\Gw_k\}_{k\in \N}$ of $\Gw$ to obtain an exhaustion of $\wb$.
Since $\Gw$ is connected we may construct $\Gw_k$    such that for each $k$, $\Gw_{k}$ is a connected component of some $D_{\varepsilon,\delta}$ (see the last paragraph in the proof of \cite[Proposition~8.2.1]{Daners}). We proceed by the following steps: 

{\bf Step 1:} 
Let $V$ be any noncompact (connected) component of $\pwr$, and consider $V$ as a $C^1$-differentiable manifold (without boundary). Then $V$ admits a $C^1$-compact exhaustion $\{U_m^V\}_{m\in \N}$ (see for example,  \cite[Proposition~2.28]{Lee} and  use \cite[Proposition~8.2.1]{Daners}).

{\bf Step 2:} 
Assume  that  $ V\in C^1$ is any  {\em noncompact} (connected) component of $\pwr$, and let $\{U_m^V\}_{m\in \N}$ be a $C^1$-compact  exhaustion of $V$.
Fix some  $U_m^V$ and consider the integral hypersurface $\partial Z_m^V$, where $Z_m^V$ is the union of all integral curves  defined by the flow
$$
\dot{\gamma}(t)=\vec{N}(t)\in S^{n-1}
$$
with initial condition $\gamma(0)\in  U_m^V$.
Here $\vec{N}$ is smooth vector field of good directions which is defined in some $n$-dimensional relative neighborhood of $U_m^V$,  see Proposition~\ref{field_good}. 
In other words, $\partial Z_m^V$ is the boundary of the union of all integral curves of the vector field $\vec{N}$  which start from $ U_m^V$. 
By Remark~\ref{rem_good}, the vector field $\vec{N}$ is never tangent to $U_m^{V}$.  Consequently, $\partial Z_m^V$ is a well defined integral hypersurface `starting' from $U_m^{V}$.

{\bf Step 3:} 
If the component $V$ is {\em compact}, then we take $V$ as a trivial exhaustion of itself, and we define $Z_1^V\subset \R^n$ to be a smooth bounded relative neighborhood of $V$ in $\Gw$.

{\bf Step 4:} 
For each $k,m$ we define the following set:
$$C_{k,m}^V:=
\left(Z_m^V\cap \big(\overline{\Gw}\setminus \Gw_k\big)\right) \cup \Gw_k.
$$
By Lemma~\ref{ballem}, for each $m\in \N$ there exist
$k_{V}(m)$
 such that for all $k\geq k_V(m)$,
  \begin{equation}\label{pathconn}
  \nabla\rho(x)\neq 0\quad \mbox{for all}\quad x\in Z_{m}^V\cap \partial \Gw_{k_V(m)}.
  \end{equation} 
\begin{rem}
	\em{
		One can visualize the set $C_{k,m}^V$ as follows: The integral surface $Z_m^V$ is a perturbed cylinder attaching $U_m^V \subset V\subset \pwr$ to its image under some projection on $\partial \Gw_k$. $C_{k,m}^V$ is then the union of the cylinder and $\Gw_k$ as depicted by the gray area in Figure~\ref{fig:1}. 
	}
\end{rem}
\begin{figure}
	\centering
	\begin{tikzpicture}

	\node at (-5.2,2) {$\partial \Gw_{\mathrm{Rob}}$};
		\node at (0,3) {$U_m^V$};
			\node at (0,1.9) {$Z_m^V$};
	\node at (0,0) {$ \Gw_{k}$};
	
		\draw [white,fill=gray, fill opacity=0.1] plot [smooth cycle] coordinates {(-2,3)
			(-1,2.8) (0,2.7) (2,2.5) (1.9,2) (2.1,1.5) (2,1)  (0,1) (-2,1)  (-1.9,1.5) (-2,2) (-2.2,2.5) (-2,3)}; 
	
		\draw [dashed,thick,black] plot [smooth] coordinates {(2,2.5) (1.9,2) (2.1,1.5) (2,1)  }; 
	
	\draw [dashed, thick,black] plot [smooth] coordinates {(-2,3) (-2.2,2.5) (-2,2) (-1.9,1.5) (-2,1)  }; 
		\draw [ thick, black,fill=gray, fill opacity=0.1] plot [smooth cycle] coordinates {(-5,0) (-2,1) (0,1) (2,1) (5,0) (0,-2) };
			\draw [thick, black] plot [smooth] coordinates {(-5,2) (-2,3) (-1,2.8) (0,2.7) (2,2.5) (5,3) }; 

	\end{tikzpicture}
	\caption{Illustration of the set $C_{k,m}^V$} \label{fig:1}
\end{figure}
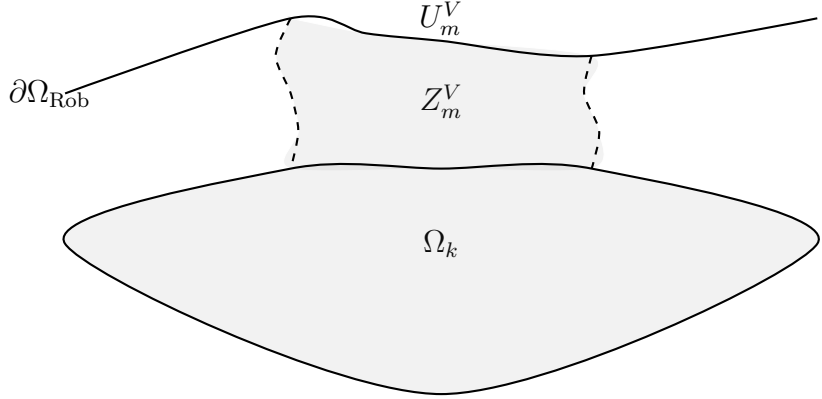
{\bf Step 5:} The last step of the construction is given by the following lemma.
\begin{lem}
Let $\{V^l\}_{l\in \N}$ be the connected components of $\pwr$ 
and let $\{U_m^{V^l}\}_{m\in \N}$ be a $C^1$ exhaustion of $V^l$.
For each $m$ let $$q_m^l:=\max\{k_{V^l}(m):1\leq l\leq m\}.$$
Then the following assertions hold.
\begin{enumerate}
		\item 	$\bigcup\limits_{m\in \N}\left(\bigcup\limits_{l=1}^{m} C_{q_m^l,m}^{V^l}\right)=\wb.$

		\item For each $m\in \N$, $\widetilde{\Gw}_m:=\bigcup\limits_{l=1}^{m} C_{q_m^l,m}^{V^l}$ is a Lipschitz connected set.
	\end{enumerate}
Hence, $\{\widetilde{\Gw}_m\}_{m\in \N}$ is a Lipschitz exhaustion of $\wb$.
\end{lem}
\begin{proof}
	(1) Follows directly  from our construction.

	(2)  For each $l,m\in \N$, $Z_{m}^{V^l}$ is $C^1$ and $\Gw_{q_m^l}$ is Lipschitz. Moreover,  $x\in \partial Z_m^{V^l}\cap \partial\Gw_{q_m^l} $ implies that $\rho(x)=\varepsilon_{q_m^l}.$
	By \cite[Remark~3.1]{BZ}, the latter equality and \eqref{pathconn} imply that  the intersection $Z_m^{V^l}\cap \overline{\Gw_{q_m^l}}$ is transversal and therefore Lipschitz.
	Moreover, the intersection $ Z_m^{V^l}\cap \big(\overline{\Gw}\setminus \Gw_{q_m^l}\big)$ is transversal by the definition of $Z_m^{V^l}$, and therefore Lipschitz. 
	
	Finally, since $Z_m^{V^l}$ is path connected we deduce that  $\widetilde{\Gw}_{l}$ is connected as well. 
\end{proof}
\begin{center}
	{\bf Acknowledgments}
\end{center}
The authors wish to thank D.~ Jerison for  a valuable discussion.
  The paper is based on part of the Ph.~D. thesis of the second author under the supervision of the first author. I.~V.  is grateful to the Technion for supporting his study.  
The  authors  acknowledge  the  support  of  the  Israel Science Foundation (grant  637/19) founded by the Israel Academy of Sciences and Humanities. 

\end{document}